\documentclass[a4paper, 11pt, reqno]{amsart}
\usepackage[top = 1in, bottom = 0.9in, left = 1in, right = 1in]{geometry}

\usepackage{amssymb,amsthm,amsmath,graphicx,xcolor,mathtools,enumerate,mathrsfs}
\usepackage{tabularx}
\usepackage[shortlabels]{enumitem} 
\usepackage[british]{babel}
\usepackage[utf8]{inputenc}
\usepackage[noadjust]{cite}

\allowdisplaybreaks

\usepackage[mathlines]{lineno}
\usepackage{etoolbox} 

\newcommand*\linenomathpatch[1]{%
	\expandafter\pretocmd\csname #1\endcsname {\linenomath}{}{}%
	\expandafter\pretocmd\csname #1*\endcsname{\linenomath}{}{}%
	\expandafter\apptocmd\csname end#1\endcsname {\endlinenomath}{}{}%
	\expandafter\apptocmd\csname end#1*\endcsname{\endlinenomath}{}{}%
}
\newcommand*\linenomathpatchAMS[1]{%
	\expandafter\pretocmd\csname #1\endcsname {\linenomathAMS}{}{}%
	\expandafter\pretocmd\csname #1*\endcsname{\linenomathAMS}{}{}%
	\expandafter\apptocmd\csname end#1\endcsname {\endlinenomath}{}{}%
	\expandafter\apptocmd\csname end#1*\endcsname{\endlinenomath}{}{}%
}

\expandafter\ifx\linenomath\linenomathWithnumbers
\let\linenomathAMS\linenomathWithnumbers
\patchcmd\linenomathAMS{\advance\postdisplaypenalty\linenopenalty}{}{}{}
\else
\let\linenomathAMS\linenomathNonumbers
\fi

\linenomathpatchAMS{gather}
\linenomathpatchAMS{multline}
\linenomathpatchAMS{align}
\linenomathpatchAMS{alignat}
\linenomathpatchAMS{flalign}
\linenomathpatch{equation}

\nolinenumbers

\usepackage{hyperref}
\hypersetup{colorlinks, linkcolor={red!50!black}, citecolor={green!50!black}, urlcolor={blue!50!black}}

\usepackage{caption}
\usepackage{subcaption}

\usepackage{cleveref}
\theoremstyle{plain}

\newtheorem{theorem}{Theorem}[section]
\crefname{theorem}{Theorem}{Theorems}

\crefname{proposition}{Proposition}{Propositions}

\crefname{corollary}{Corollary}{Corollaries}

\newtheorem{lemma}[theorem]{Lemma}
\crefname{lemma}{Lemma}{Lemmas}

\crefname{conjecture}{Conjecture}{Conjectures}

\crefname{observation}{Observation}{Observations}

\newtheorem*{claim}{Claim}
\crefname{claim}{Claim}{Claims}

\theoremstyle{definition}
\newtheorem{definition}[theorem]{Definition}
\crefname{definition}{Definition}{Definitions}

\crefname{construction}{Construction}{Constructions}

\crefname{question}{Question}{Questions}

\crefname{section}{Section}{Sections}
\crefname{figure}{Figure}{Figures}
\numberwithin{equation}{section}
\crefname{enumi}{}{}

\definecolor{VividOrange}{HTML}{F15918}
\definecolor{NeonBlue}{HTML}{1F51FF}

\DeclarePairedDelimiter\abs{\lvert}{\rvert}
\DeclarePairedDelimiter\norm{\lVert}{\rVert}
\renewcommand{\subset}{\subseteq}

\def\leq{\leqslant}

\def\geq{\geqslant}

\newcommand{\II}{\mathcal{I}}

\newcommand{\NN}{\mathbb{N}}
\newcommand{\ZZ}{\mathbb{Z}}

\newcommand{\PM}{\mathcal{P}}
\newcommand{\CM}{\mathcal{C}}

\newcommand{\Reg}[1]{\ensuremath{\mathrm{Reg}({#1})}}
\newcommand{\ExpM}[1]{\ensuremath{\mathrm{Exp_1}({#1})}}
\newcommand{\ExpL}[1]{\ensuremath{\mathrm{Exp_2}({#1})}}
\newcommand{\Def}[1]{\ensuremath{\mathrm{Def}({#1})}}

\newenvironment{proofclaim}[1][Proof of the claim]{\begin{proof}[#1]}{\end{proof}}

\usepackage{tikz}

\title[]{Asymptotically Enumerating Independent Sets \\ in Regular $k$-Partite $k$-Uniform Hypergraphs}
\date{\today}
\author[Arras, Garbe, and Joos]{Patrick Arras, Frederik Garbe, and Felix Joos}
\thanks{The research leading to these results was partially supported by the Deutsche \mbox{Forschungsgemeinschaft} (DFG, German Research Foundation) -- 428212407.}
\address{Universität Heidelberg,
	Institut für Informatik,
	Im Neuenheimer Feld 205,
	69120 Heidelberg, Germany}
\email{\{arras,garbe,joos\}@informatik.uni-heidelberg.de}

\begin{document}
\begin{abstract}
The number of independent sets in regular bipartite expander graphs can be efficiently approximated by expressing it as the partition function of a suitable polymer model and truncating its cluster expansion. While this approach has been extensively used for graphs, surprisingly little is known about analogous questions in the context of hypergraphs.

In this work, we apply this method to asymptotically determine the number of independent sets in regular $k$-partite $k$-uniform hypergraphs which satisfy natural expansion properties. The resulting formula depends only on the local structure of the hypergraph, making it computationally efficient. In particular, we provide a simple closed-form expression for linear hypergraphs.
\end{abstract}

\maketitle
\thispagestyle{empty}
\vspace{-0.4cm}

\section{Introduction}
Asymptotic enumeration of independent sets has a long history. One of the earliest results is the famous theorem due to Korshunov and Sapozhenko~\cite{KS83} that the $n$-dimensional discrete hypercube contains $(1 \pm o(1)) \cdot 2 \sqrt e \cdot 2^{2^{n-1}}$ independent sets. Note that this graph is bipartite with partition classes of order $2^{n-1}$, so roughly $2 \cdot 2^{2^{n-1}}$ independent sets already arise by considering only subsets of either partition class. In fact, it turns out that almost every independent set is close to being a subset of a partition class, that is having only a very small number of \emph{defect vertices} in the other class. Roughly speaking, this stems from the fact that every defect vertex~$v$ included in an independent set $I$ significantly reduces the variability of the intersection of $I$ with the other partition class by preventing all $n$ neighbours of $v$ from being selected.

Over the years, this result has inspired extensive research, leading to significant generalizations of the initial setting. Most often, these rephrase the enumeration of independent sets as an evaluation of the \emph{partition function} of a certain \emph{polymer model}, the \emph{hard-core model} from statistical physics, at \emph{fugacity} $1$. Generalizations then arise from varying the fugacity or modifying the polymer model. Translating back to the language of graphs, this amounts to computing a weighted sum of independent sets or counting more general objects such as proper $q$-colourings or homomorphisms to a fixed graph. One of the most comprehensive studies to date is due to Jenssen and Keevash~\cite{JK20}, who do not only allow for most of these variations, but also replace the hypercube with the more general class of discrete tori of even sidelength.

On the technical side of things, the four decades since Korshunov and Sapozhenko's initial result have also seen numerous improvements of the underlying methods. Perhaps the most impactful contributions were the introduction of a \emph{graph container method} by Sapozhenko~\cite{Sap89} to simplify the proof, and the introduction of entropy tools by Kahn~\cite{Kah01}, both of which have become ubiquitous ever since. Notably, they combine very well to form a useful machinery, as exemplified by the works of Peled and Spinka~\cite{PS23} as well as Kahn and Park~\cite{KP20}. Another recent result due to Helmuth, Perkins, and Regts~\cite{HPR20} demonstrates that the \emph{contour models} from Pirogov-Sinai theory can be applied yielding a similar effect.

While most early results concerned rather specific graph classes, it quickly became apparent that only a few properties of the graphs in question are actually relevant, namely regularity, bipartiteness, and a quantifiable (albeit small) expansion. One can then group the objects of interest by identifying their intersection with one of the two partition classes as their \emph{defect set} and count the objects grouped according to their defect set. Trying to relax the required graph properties as much as possible, a significant proportion of research is now devoted to graph classes so general that the final count cannot be brought into a closed-form expression. Instead, the statements obtained provide a fast algorithm for calculating an arbitrarily precise approximation of the desired number from the local structure of the specific graph. Such an algorithm is called a \emph{fully polynomial-time approximation scheme} (FPTAS) and constitutes the most detailed description one can hope for without requiring more information about the local structure of the graph in question. Recent results include many different settings for which an FPTAS for the number of independent sets in bipartite expanders have been found~\cite{CP20,CDK+23,CGG+21,CGS+22,FGK+23,GGS21,GGS22,JKP20,JMP24,JPP23,LLL+19,LLL+22}.

Obtaining similar results for non-bipartite graphs appears to be very difficult, as previous work by the first and third author~\cite{AJ23} indicates; Jenssen and Keevash~\cite{JK20} explicitly ask for the number of independent sets in $\ZZ_3^n$. The same question can of course be asked for hypergraphs. Here, an FPTAS for the number of (weak) independent sets is only known if the maximum degree is bounded by a constant depending on the uniformity~\cite{FGW+22,HWY23}. Otherwise, finding an approximation is NP-hard~\cite{BGG+19}. Our results therefore concern counting (weak) independent sets in hypergraphs with the natural generalization of bipartiteness (our hypergraphs are $k$-uniform and $k$-partite). As in the graph case, we assume some expansion properties. With this, we are able to approximate said number much closer than the best upper bound available in the non-$k$-partite setting~\cite{BS18,CPS+22}.

While our main result (\cref{thm: main}) is formulated in the language of cluster expansion and significantly more general, it gives rise to a straightforward method of calculating the number of independent sets if the degree is sufficiently large. Before we can formulate this precisely, we need to address how the aforementioned hypergraph properties look in our setting. 

Let $k \geq 2$. A hypergraph is called a \emph{$k$-graph} if every edge contains exactly $k$ vertices. It is called \emph{$k$-partite} if its vertex set can be partitioned into $k$ subsets (called \emph{partition classes}) such that every edge contains one vertex of each partition class. This is the natural way to move from bipartite ($2$-)graphs to the case of general $k$. We call a hypergraph \emph{$r$-regular} if every vertex is contained in exactly $r$ edges.

We write $\gamma_k \coloneqq \frac {2^{k-1}}{2^{k-1}-1} > 1$ for ease of notation. For a vertex subset $S$ of a $k$-graph $G$, we define its \emph{neighbourhood} by $N_G(S) \coloneqq \big( \bigcup_{e \in E(G), e \cap S \neq \emptyset} e \big) \setminus S$ and omit the subscript if $G$ is clear from the context.

\begin{definition}\label{def: expansion properties}
Let $k \geq 2$ and $G$ be an $r$-regular $k$-partite $k$-graph with vertex partition $\mathcal Z$, where each part has order $n$. Given $t,b \in \NN$ and $\alpha,\beta > 0$, we define the following properties of~$G$: \smallskip

\begin{tabular}{ll}
	\Reg{t}: & $r \geq \frac 1t \log_{\gamma_k} n$.\\
	\ExpM{\alpha}: & For every $S \subset Z \in \mathcal Z$ with $\abs S \leq r$, we have $\abs {N(S)} \geq (k-1-\alpha)r \abs S$.\\
	\ExpL{\beta}: & For every $S \subset Z \in \mathcal Z$ with $\abs S \leq \beta \frac nr$, we have $\abs {N(S)} \geq (k-2+\beta)r \abs S$.\\
	\Def{b}: & For every $I \in \II(G)$, there is some $Z \in \mathcal Z$ such that $\abs {I \cap Z} \leq b$.
\end{tabular}
\end{definition}

Note that \Reg{t} establishes a logarithmic relationship between the degree and the order of the hypergraph. This is very much akin to results on discrete tori such as the hypercube, which exhibit the same behaviour. Given a subset $S$ of some partition class $Z$, both expansion conditions $\mathrm{Exp_1}$ and $\mathrm{Exp_2}$ provide a lower bound on $\abs {N(S)}$. Note that the larger the set $S$, the smaller the required relative expansion. Lastly, property \Def{b} is necessary to allow for the aforementioned identification of defect sets. All of these conditions are typically satisfied in pseudorandom graphs and can also be expected for random regular $k$-partite $k$-graphs of logarithmic degree.

With this, we can formulate our result on enumerating \emph{independent sets}, that is sets that do not contain any edge as a subset. We denote the set consisting of all independent sets in $G$ by $\II(G)$. To obtain a closed-form expression for $\abs {\II(G)}$, we have to assume $G$ to be \emph{linear}, that is that every pair of distinct edges shares at most one vertex. The following is an immediate corollary of our main result. By $\alpha_k$, we denote a (small) positive number that only depends on~$k$. We omit the exact value for now, but make it explicit later on.

\begin{theorem}\label{thm: t=1}
Let $k \geq 3$ and $\beta > 0$. Furthermore, let $b(n)$ be a function with $b(n) = o(n)$. Then for any $\rho > 0$, there is $n_0 \in \NN$ such that the following holds for every $n \geq n_0$:

Let $G$ be a linear $r$-regular $k$-partite $k$-graph with vertex partition $\mathcal Z$, where each part has order $n$. If $G$ satisfies \Reg{1}, \ExpM{\alpha_k}, \ExpL{\beta}, and \Def{b(n)}, then
\[
	\abs {\II(G)}
	= (1 \pm \rho) \cdot k \cdot 2^{(k-1)n} \cdot \exp \left( n \gamma_k^{-r} \right)
	\,.
\]
\end{theorem}

Note that while our proof requires $k \geq 3$, the same statement is also true for $k = 2$. In fact, this follows from a more general result of Jenssen, Perkins, and Potukuchi~\cite{JPP23}. We discuss the details in our concluding remarks. Also note that $r \geq \log_{\gamma_k} n$ by \Reg{1} implies that the final exponential term contributes at most a factor of $e$. Moreover, if $r$ is larger, even if only by a constant factor, the term tends to $1$ as $n \to \infty$, which in turn implies that essentially all independent sets have an empty intersection with one partition class. The primary advantage of our main theorem (\cref{thm: main}) is that it provides a similar bound while requiring only \Reg{t} for some $t \in \NN$. It turns out, however, that the larger $t$ becomes, the more terms we have to include in the argument of the exponential. This reflects the fact that with lower degree, the cost of including a vertex in the independent set (that is, the amount of neighbours now prevented from being included) decreases, so larger defect sets arise.

As the defect sets become larger, their typical structure changes as well. In the hypercube result of~\cite{KS83}, multiple defect vertices of the same independent set are typically at distance greater than $2$ from each other and can thus be chosen essentially independently. The same is true for the setting of \cref{thm: t=1}. For larger $t$, however, larger defect sets mean that there is now also a significant proportion of independent sets with defect vertices whose neighbourhoods intersect.

The paper is structured as follows. In \cref{sec: state main} we introduce the necessary notation to formalize our cluster expansion approach and state our main theorem. We also discuss a few applications such as \cref{thm: t=1}. We then collect some tools and further preliminaries in \cref{sec: tools nota}. Afterwards, we prove our main theorem in \cref{sec: proof main}, while for now assuming the key technical ingredient (the verification of the Kotecký-Preiss condition in \cref{lem: KP condition holds}). Finally, \cref{sec: KP} is dedicated to the proof of \cref{lem: KP condition holds}.

\section{Statement of the main theorem}\label{sec: state main}
\subsection{Cluster expansion}\label{subsec: ce}
We use the well-known cluster expansion approach from statistical physics, mostly following the notation of~\cite{JK20}. This section first introduces the necessary concepts in general and then defines the concrete polymer model used throughout the paper. Such a \emph{polymer model} $(\PM, \sim, w)$ consists of a set $\PM$ of so-called \emph{polymers}, together with a \emph{compatibility} relation $\sim$ on $\PM$ and a weight function $w \colon \PM \to [0, \infty)$. The \emph{order} of a polymer $S$ is simply $\abs S$, its number of elements. A set $\mathcal S \subset \PM$ of polymers is called \emph{compatible} if the polymers in $\mathcal S$ are pairwise compatible. The \emph{partition function} of the polymer model is then given by
\begin{align}\label{eq: partition function}
	\Xi_{(\PM, \sim, w)} 
	\coloneqq \sum_{\substack{\mathcal S \subset \PM\\\text{$\mathcal S$ compatible}}} \prod_{S \in \mathcal S} w(S)
	\,.
\end{align}

For a non-empty vector $\Gamma$ of not necessarily distinct polymers, we let its \emph{incompatibility graph}~$H_\Gamma$ be the graph with the polymers of $\Gamma$ as vertices and edges between two polymers $S, T \in \Gamma$ if $S \nsim T$. A \emph{cluster} is a vector $\Gamma$ of polymers such that $H_\Gamma$ is connected. We denote its \emph{length} (that is, the number of, not necessarily distinct, polymers in the cluster) as $\abs \Gamma$ and obtain its \emph{size} $\norm \Gamma \coloneqq \sum_{S \in \Gamma} \abs S$ by adding up the orders of all polymers in the cluster. We also denote the infinite set of all clusters as $\CM_{(\PM, \sim, w)}$.

Recall now that the \emph{Ursell function} of a connected graph $H$ is
\begin{align*}
	\phi(H)
	\coloneqq \frac 1{\abs {V(H)}!} \sum_{\substack{\text{spanning, connected}\\\text{subgraphs $F \subset H$}}} (-1)^{\abs {E(F)}}
	\,.
\end{align*}
With this, we define the weight of a cluster $\Gamma$ as 
\begin{align}\label{eq: cluster weights}
	w(\Gamma) 
	\coloneqq \phi(H_\Gamma) \prod_{S \in \Gamma} w(S)
	\,.
\end{align}
By~\cite[Proposition 5.3]{FV17}, we have
\begin{align}\label{eq: cluster expansion} 
	\log \Xi_{(\PM, \sim, w)}
	= \sum_{m = 1}^\infty \sum_{\substack{\Gamma \in \CM_{(\PM, \sim, w)} \\ \norm \Gamma = m}} w(\Gamma)
	\,,
\end{align}
if the right-hand side converges absolutely. This sum is called the \emph{cluster expansion} of the polymer model $(\PM, \sim, w)$ and, provided it converges, can be truncated to yield a good approximation of $\log \Xi_{(\PM, \sim, w)}$.

\subsection{Polymer model}\label{subsec: polymer model}
For a hypergraph $G$ and a vertex subset $S \subset V(G)$, we write $\overline S \coloneqq V(G) \setminus S$ for the complement of $S$ and $G[S]$ for the subgraph of $G$ induced by $S$. In order to define a polymer model suitable for our task of counting independent sets, we use the notion of $2$-linkedness: For a $k$-partite $k$-graph $G$ and a partition class $Z$, we consider the $2$-graph $Z_2$ on the vertex set $Z$ with $vu \in E(Z_2)$ if and only if $N_G(\{ v \}) \cap N_G(\{ u \}) \neq \emptyset$. A non-empty set $S \subset Z$ is called \emph{$2$-linked} if $Z_2[S]$ is connected.

Furthermore, for a $k$-partite $k$-graph $G$ and a subset $S$ of one partition class, we define the \emph{link graph} of $S$ as the $(k-1)$-graph $L_G(S)$ on the vertex set $N_G(S)$ with edge set
\[
	E(L_G(S))
	\coloneqq \{ e \setminus S \colon \text{$e \in E(G)$ with $e \cap S \neq \emptyset$} \}
	\,.
\]

Now we can describe the polymer model we will use. For a $k$-partite $k$-graph $G$, any partition class $Z$, and a number $b \in \NN$, we consider the polymer model given by 
\begin{align}\label{eq: polymer}
	\PM_{Z,b} \coloneqq \{ S \subset Z \colon \text{$S$ is $2$-linked and $1 \leq \abs S \leq b$} \}
	\,,
\end{align}
the compatibility relation defined by $S \sim T$ if and only if $N_G(S) \cap N_G(T) = \emptyset$, and the polymer weights of 
\begin{align}\label{eq: polymer weights}
	w(S) 
	\coloneqq \abs {\II(L_G(S))} \cdot 2^{-\abs {N_G(S)}}
	\,.
\end{align}
For the sake of simplicity, we write $\Xi_{Z,b} \coloneqq \Xi_{(\PM_{Z,b}, \sim, w)}$ and $\CM_{Z,b} \coloneqq \CM_{(\PM_{Z,b}, \sim, w)}$ for the partition function and the clusters of this polymer model.

\subsection{Main theorem}
We have now introduced the setup to be able to state our main theorem. Afterwards, we discuss how to interpret the statement and how to apply it, for example to prove \cref{thm: t=1}. We set 
\begin{align}\label{eq: alpha0}
	\alpha_{k,t}
	\coloneqq \frac 12 \min \left\lbrace \frac {\log_2 \gamma_k}{e^{2t}} ; \frac {(k-1)(1 - \log 2)\log \gamma_k}{\log (2^{k-1} - 1) + \log \gamma_k} \right\rbrace
\end{align}
and remark that $0 < \alpha_{k,t} < 1$ for every $k \geq 2$ and $t \in \NN$.

\begin{theorem} \label{thm: main}
Let $k \geq 3$, $t \in \NN$, and $\beta > 0$. Furthermore, let $b(n)$ be a function with $b(n) = o(n)$. Then for any $\rho > 0$, there is $n_0 \in \NN$ such that the following holds for every $n \geq n_0$:
	
Let $G$ be an $r$-regular $k$-partite $k$-graph with vertex partition $\mathcal Z$, where each part has order $n$. If $G$ satisfies \Reg{t}, \ExpM{\alpha_{k,t}}, \ExpL{\beta}, and \Def{b(n)}, then
\[
	\abs {\II(G)}
	= (1 \pm \rho) \cdot 2^{(k-1)n} \cdot \sum_{Z \in \mathcal Z} \exp \left( \sum_{m = 1}^t \sum_{\substack{\Gamma \in \CM_{Z,t}\\\norm \Gamma = m}} w(\Gamma) \right)
	\,.
\]
\end{theorem}

Notably, neither the parameter $\beta$ nor the function $b(n)$ occur in the formula of \cref{thm: main}, but only influence the choice of $n_0$, which we do not make explicit. Likewise, both occurrences of the parameter $t$ only determine the cluster size (and consequently, maximum polymer order) at which to truncate the sum in the argument of the exponential. Additional terms could still be computed and included, but are guaranteed to vanish into the error term of $1 \pm \rho$ and thus do not affect the asymptotics of $\abs {\II(G)}$.

In other words, \cref{thm: main} asserts that for any hypergraph $G$ satisfying the required properties, the asymptotic number of independent sets is governed solely by the clusters up to size $t$. The structure of these clusters, however, is a local property of $G$: Since the polymers in a cluster $\Gamma \in \CM_{Z,t}$ are $2$-linked and the incompatibility graph $H_\Gamma$ of $\Gamma$ is connected, the set $V(\Gamma) \coloneqq \bigcup_{S \in \Gamma} S$ is $2$-linked as well. Moreover, we have $\abs {V(\Gamma)} \leq \norm \Gamma$. In order to determine all clusters of size at most a constant $t$ that contain some fixed vertex $v$, it therefore suffices to only consider vertices at distance at most $2(t-1)$ from $v \in Z$, which can be done in time polynomial in $n$.

For $t = 1$, this comes down to only examining the link graph $L_G(\{ v \})$ of $v$ itself. In general $r$-regular $k$-partite $k$-graphs, the vertex $v$ might share multiple edges with any other vertex $u$, which gives rise to a multitude of possible link graphs. If we require $G$ to be linear, however, this cannot occur and $L_G(\{ v \})$ is a perfect matching of $r$ pairwise disjoint edges of uniformity $k-1$. This very simple local structure allows for the deduction of \cref{thm: t=1} from \cref{thm: main}.

\begin{proof}[Proof of \cref{thm: t=1}]
Setting $\alpha_k \coloneqq \alpha_{k,1}$, any $G$ in \cref{thm: t=1} satisfies the requirements of \cref{thm: main} with $t = 1$. It remains to determine all clusters $\Gamma \in \CM_{Z,1}$ with $\norm \Gamma = 1$ and compute $w(\Gamma)$. Since polymers are non-empty by (\ref{eq: polymer}), the only possibility is $\Gamma = (\{ v \})$ for $v \in Z$. Naturally, there are exactly $\abs Z = n$ such clusters. 

The incompatibility graph of a cluster $\Gamma = (\{ v \})$ is a single vertex, so $\phi(H_{(\{ v \})}) = 1$. Since $L_G(\{ v \})$ is a perfect matching of $r$ pairwise disjoint edges of uniformity $k-1$, we have $\abs {\II(L_G(\{ v \}))} = (2^{k-1} - 1)^r$ and $2^{-\abs {N(\{ v \})}} = 2^{-(k-1)r}$. Altogether, we obtain $w(\Gamma) = \gamma_k^{-r}$ from (\ref{eq: cluster weights}) and (\ref{eq: polymer weights}).

The argument of the exponential in \cref{thm: main} thus evaluates to $n\gamma_k^{-r}$. Since this is true for any $Z \in \mathcal Z$, the desired $\abs {\II(G)} = (1 \pm \rho) \cdot k \cdot 2^{(k-1)n} \cdot \exp (n\gamma_k^{-r})$ follows.
\end{proof}

If the degree of $G$ is lower, we need to calculate more terms and therefore require additional information about the local structure of $G$. To illustrate this, we present the case $t = 2$. While linearity still guarantees that the link graph of any singular vertex $v$ is a matching of $r$ edges, a cluster of size $2$ now contains a second vertex $u$ whose neighbourhood has to intersect that of $v$ in at least one vertex. If we prevent larger intersections, we can still obtain a simplified formula.

A natural property which prevents such intersections is \emph{high girth} with respect to loose cycles: For $k \geq 2$ and $\ell \geq 3$, a \emph{loose $\ell$-cycle} in a $k$-graph $G$ is a sequence of $(k-1)\ell$ distinct vertices $v_1, \ldots, v_{(k-1)\ell}$ such that $\{ v_{(k-1)(i-1) + 1}, \ldots, v_{(k-1)i + 1} \} \in E(G)$ for every $i \in \{ 1, \ldots, \ell \}$, where subscripts are taken modulo $(k-1)\ell$. The minimum $\ell$ such that $G$ contains a loose $\ell$-cycle is called the \emph{girth} of $G$. If a $k$-graph contains no loose cycle, we say that it has girth $\infty$. With this, one can prove the following property.

\begin{lemma}\label{lem: high girth determines local structure}
Let $k \geq 3$ and $G$ be a linear $k$-partite $k$-graph of girth at least $5$. Then for any partition class $Z$ and any edge $vu \in E(Z_2)$, we have $\abs {N_G(\{ v \}) \cap N_G(\{ u \})} = 1$.
\end{lemma}

We defer its proof to the appendix. \cref{lem: high girth determines local structure} shows that requiring $G$ to have girth at least~$5$ is sufficient to determine the local structure of $G$ for clusters of size $2$. The following theorem demonstrates how to use this knowledge to calculate the terms of the cluster expansion for $t = 2$.

\begin{theorem}\label{thm: t=2}
Let $k \geq 3$ and $\beta > 0$. Furthermore, let $b(n)$ be a function with $b(n) = o(n)$. Then for any $\rho > 0$, there is $n_0 \in \NN$ such that the following holds for every $n \geq n_0$:

Let $G$ be a linear $r$-regular $k$-partite $k$-graph of girth at least $5$ with vertex partition $\mathcal Z$, where each part has order $n$. If $G$ satisfies \Reg{2}, \ExpM{\alpha_{k,2}}, \ExpL{\beta}, and \Def{b(n)}, then
\[
	\abs {\II(G)}
	= (1 \pm \rho) \cdot k \cdot 2^{(k-1)n} \cdot \exp \left( n \gamma_k^{-r} + \frac {k-1}4 r^2n \gamma_k^{-2r} \left( \frac {r-1}{r} \gamma_k^2 \left( 1 + \left( \frac {2^{k-2} - 1}{2^{k-2}} \right)^2 \right) - 2 \right) \right)
	\,.
\]
\end{theorem}

\begin{proof}
Any $G$ in \cref{thm: t=2} satisfies the requirements of \cref{thm: main} with $t = 2$. The calculation for clusters of size $1$ is identical to the proof of \cref{thm: t=1}. It remains to determine all clusters $\Gamma \in \CM_{Z,2}$ with $\norm \Gamma = 2$ and compute $w(\Gamma)$. These split into two distinct types.

Firstly, there are clusters $\Gamma = (\{ v \}, \{ u \})$ for $v, u \in Z$ (not necessarily distinct). Given one of the $n$ vertices $v \in Z$, we know that $u$ must be a neighbour of one of the $(k-1)r$ vertices $x \in N(\{ v \})$. Additionally, each such $u$ cannot be a neighbour of another $x' \in N(\{ v \}) \setminus \{ x \}$, otherwise $x, x' \in N(\{ v \}) \cap N(\{ u \})$ in contradiction to \cref{lem: high girth determines local structure}. Since each $x \in N(\{ v \})$ has exactly $r$ neighbours in $Z$, we conclude that there are $n(k-1)r^2$ clusters of the first type (note that order matters). Their incompatibility graph is a single edge with $\phi(H_{(\{ v \}, \{ u \})}) = -\frac 12$. The weight of a polymer of order $1$ has already been computed as $\gamma_k^{-r}$ in the proof of \cref{thm: t=1}, so by using (\ref{eq: cluster weights}), we get
\[
	\sum_{\substack{\Gamma = (\{ v \}, \{ u \}) \in \CM_{Z,2} \\ v,u \in Z}} w(\Gamma)
	= -\frac 12 n(k-1)r^2 \gamma_k^{-2r}
\]
as the contribution of all clusters of the first type to the argument of the exponential.

Secondly, there are clusters $\Gamma = (\{ v, u \})$ for distinct $v, u \in Z$. Given one of the $n$ vertices $v \in Z$, we again select any of the $(k-1)r$ vertices $x \in N(\{ v \})$ and $u$ as a neighbour of $x$ in $Z$. Again, each $u$ can only be a neighbour of one $x \in N(\{ v \})$ by \cref{lem: high girth determines local structure}. Since we must not pick $v$ itself this time, there are only $r-1$ choices for $u$ given $x$. Moreover, the order of vertices in a polymer does not matter, which halves the number of clusters to $\frac 12 n(k-1)r(r-1)$ of the second type.

Their incompatibility graph is a single vertex with $\phi(H_{(\{ v, u \})}) = 1$. According to \cref{lem: high girth determines local structure}, the fact that $G$ has girth at least $5$ guarantees that the link graph of a polymer $\{ v, u \}$ can be obtained by intersecting the link graphs of $v$ and $u$ (in both cases, $r$ disjoint edges of uniformity $k-1$) in a single vertex. This yields $2r-2$ disjoint edges of uniformity $k-1$ plus two edges of uniformity $k-1$ intersecting in a single vertex. By (\ref{eq: polymer weights}), the disjoint edges again contribute $\gamma_k^{-(2r-2)}$ to $w(\{ v, u \})$. The two intersecting edges of uniformity $k-1$ contribute $(2^{k-2})^2 + (2^{k-2} - 1)^2$ to $\abs {\II(L_G(\{ v, u \}))}$ and $2^{-2k+3}$ to $2^{-\abs {N(\{ v, u \})}}$, so $\frac 12 + \frac 12 \Big( \frac {2^{k-2} - 1}{2^{k-2}} \Big)^2$ to $w(\{ v, u \})$. Therefore by (\ref{eq: cluster weights}), the entirety of clusters of the second type contributes 
\[
	\sum_{\substack{\Gamma = (\{ v, u \}) \in \CM_{Z,2} \\ v,u \in Z}} w(\Gamma)
	= \frac 12 n(k-1)r(r-1) \cdot \gamma_k^{-(2r-2)} \cdot \left( \frac 12 + \frac 12 \left( \frac {2^{k-2} - 1}{2^{k-2}} \right)^2 \right)
\]
to the argument of the exponential.

We observe that these calculations are identical for every $Z \in \mathcal Z$. Plugging them into the formula of \cref{thm: main} yields the desired result.
\end{proof}

It is worth noting that with \cref{thm: t=2} only requiring $r \geq \frac 12 \log_{\gamma_k} n$, the $\exp(n \gamma_k^{-r})$-term representing clusters of size $1$ can now grow as fast as $\exp(\sqrt n)$. Consequently, the approximately $k \cdot 2^{(k-1)n}$ independent sets with zero defects, depending on $r$, may only account for an exponentially small fraction of all independent sets.

\section{Tools and notation}\label{sec: tools nota}
Before we can begin with the proof of our main theorem, we first need to fix some more notation and collect a few additional tools.

All our arguments will eventually examine the limit as $n \to \infty$. For the sake of brevity, we utilize the Landau notation $o(h(n))$ to denote any function $j(n)$ satisfying $j(n) / h(n) \to 0$ as $n \to \infty$. We also write $j(n) = (1 \pm \rho) h(n)$ to express that $(1 - \rho)h(n) \leq j(n) \leq (1 + \rho)h(n)$ holds.

Unless stated otherwise, $\log$ is the natural logarithm. We omit rounding in our notation when it does not affect the argument.

\subsection{Expansion properties}\label{subsec:expansion}
We observe a few consequences of the expansion properties we introduced in \cref{def: expansion properties}. Firstly, we note that in order for $\mathrm{Exp_1}$ to make sense, we need to assume that $r$ is not too large. The following lemma captures this.

\begin{lemma}\label{lem: ExpM saves KP3}
Let $k \geq 3$ and $0 < \alpha < 1$. Suppose $G$ is an $r$-regular $k$-partite $k$-graph with vertex partition $\mathcal Z$, where each part has order $n \geq 3$. If $G$ satisfies \ExpM{\alpha}, then $r \leq \sqrt {2n}$.
\end{lemma}

\begin{proof}
Consider any $Z \in \mathcal Z$ and an arbitrary $S \subset Z$ with $\abs S = \min \{ r, n \}$. Since $N(S) \subset \overline Z$, we observe that \ExpM{\alpha} together with $\alpha < 1$ and $k \geq 3$ implies
\[
	r \min \{ r, n \}
	= r \abs S
	\overset {\mathrm{Exp_1}}\leq \frac {\abs {N(S)}}{k-1-\alpha}
	\overset {\alpha < 1}\leq \frac {\abs {\overline Z}}{k-2}
	= \frac {k-1}{k-2} \cdot n
	\overset {k \geq 3}\leq 2n
	\,.
\]
For $n \leq r$, this simplifies to $r \leq 2$, which contradicts $3 \leq n \leq r$. So $r \leq n$ must hold and we obtain $r^2 \leq 2n$, which implies the desired $r \leq \sqrt {2n}$.
\end{proof}

Secondly, we observe that $\mathrm{Exp_2}$ can be used to establish a lower bound for the relative expansion of sets $S$ even larger than $\beta \frac nr$.

\begin{lemma}\label{lem: ExpL relative expansion}
Let $k \geq 3$, $b \in \NN$, and $\beta > 0$. Suppose $G$ is an $r$-regular $k$-partite $k$-graph with vertex partition $\mathcal Z$, where each part has order $n$. If $G$ satisfies \ExpL{\beta}, then for every $S \subset Z \in \mathcal Z$ with $\abs S \leq b$, we have $\abs {N(S)} \geq \min \left\lbrace r ; \frac {\beta n}b \right\rbrace \abs S$.
\end{lemma}

\begin{proof}
Let $Z \in \mathcal Z$ and $S \subset Z$ with $\abs S \leq b$ be arbitrary. If $\abs S \leq \beta \frac nr$, we use \ExpL{\beta} directly to obtain $\abs {N(S)} \geq (k-2+\beta)r \abs S \geq r \abs S$. If $\abs S > \beta \frac nr$, we select an arbitrary subset $S' \subset S$ of order $\abs {S'} = \beta \frac nr$. Invoking \ExpL{\beta} for $S'$, we obtain
\[
	\abs {N(S)}
	\geq \abs {N(S')}
	\overset {\mathrm{Exp_2}}\geq (k-2+\beta)r \abs {S'}
	\overset {k \geq 3}\geq r \abs {S'}
	= \beta n
	\overset{\abs S \leq b}\geq \frac {\beta n}{b} \abs S
	\,.
\]
This proves the lemma.
\end{proof}

\subsection{Kotecký-Preiss condition}
When introducing the cluster expansion in (\ref{eq: cluster expansion}), we already hinted at the fact that it can be truncated to obtain a good approximation of the partition function of a polymer model, provided it converges. Establishing said convergence can be achieved via the following statement, also known as the Kotecký-Preiss condition.

\begin{theorem}[Kotecký-Preiss condition~\cite{KP86}]\label{thm: KP}
Consider a polymer model $(\PM, \sim, w)$ and two functions $f, g \colon \PM \to [0, \infty)$. Suppose that for all $S \in \PM$, we have
\[
	\sum_{\substack{T \in \PM \\ T \nsim S}} w(T) \exp \big( f(T) + g(T) \big)
	\leq f(S)
	\,.
\]
Then the cluster expansion of $(\PM, \sim, w)$ converges absolutely. Furthermore, for any $S \in \PM$, we have 
\[
	\sum_{\substack{\Gamma \in \CM_{(\PM, \sim, w)} \\ \Gamma \nsim S}} \abs {w(\Gamma)} \exp \left( \sum_{T \in \Gamma} g(T) \right)
	\leq f(S)
	\,,
\]
where the outer sum is over all clusters $\Gamma \in \CM_{(\PM, \sim, w)}$ that contain at least one polymer $T \in \PM$ that satisfies $T \nsim S$.
\end{theorem}

Note that this condition is easiest to satisfy for $g \equiv 0$. When using non-zero~$g$, however, \cref{thm: KP} provides a tail bound, which will be useful for our proof of \cref{thm: main}.

\subsection{2-linkedness}
Since we use this observation throughout the paper, it is worth noting here that when starting from a fixed vertex, the number of $2$-linked sets containing it does not grow too fast. In order to formalize this, we require the following observation from~\cite{Gal19}.

\begin{lemma}[{\cite{Gal19}, see Lemma~2.2}]\label{lem: bounded max deg fixed size}
Let $H$ be a graph of maximum degree $\Delta$, $v \in V(H)$, and $s \in \NN$. Then there are at most $e(e\Delta)^{s-1}$ vertex subsets $S \subset V(H)$ with $v \in S$ and $\abs S = s$ such that $H[S]$ is connected.
\end{lemma}

\begin{proof}
Note that the statement trivially holds for $s = 1$, so we may assume $s \geq 2$. As pointed out in~\cite{Gal19}, the number of vertex subsets that contain $v$, have order~$s$, and induce a connected subgraph can be bounded from above by the number of rooted subtrees of the infinite $\Delta$\nobreakdash-branching rooted tree that have order $s$. This number is
\[
	\frac {\binom{\Delta s}{s}} {(\Delta - 1)s + 1}
	= \frac {\binom{\Delta s}{s-1}} s
	\leq \binom{\Delta s}{s-1}
	\leq \left( e \frac {\Delta s}{s-1} \right)^{s-1}
	= (e\Delta)^{s-1} \left( 1 + \frac 1{s-1} \right)^{s-1}
	\leq e(e\Delta)^{s-1}
	\,,
\]
which finishes the proof.
\end{proof}

Applied to $2$-linkedness and our setting, this implies the following.

\begin{lemma}\label{lem: 2-linked fixed size}
Let $k \geq 3$ and $G$ be an $r$-regular $k$-partite $k$-graph. Then for every partition class~$Z$, every vertex $v \in Z$, and every $s \in \NN$, there are at most $e((k-1)er^2)^{s-1}$ $2$-linked sets $S \subset Z$ with $v \in S$ and $\abs S = s$.
\end{lemma}

\begin{proof}
Note that in $G$, every $v \in Z$ has at most $(k-1)r$ neighbours, all of which belong to $\overline Z$. Also note that every $x \in \overline Z$ has at most $r$ neighbours in $Z$. This shows that the maximum degree of $Z_2$ is at most $(k-1)r^2$. Hence, \cref{lem: bounded max deg fixed size} guarantees that there are at most $e((k-1)er^2)^{s-1}$ subsets $S \subset V(Z_2)$ with $v \in S$ and $\abs S = s$ such that $Z_2[S]$ is connected. By definition, these are exactly the $2$-linked sets in question.
\end{proof}

\section{Proof of the main theorem}\label{sec: proof main}
Our first observation is how the partition function of the polymer model introduced in \cref{sec: state main} relates to the number of independent sets in question. We require an easy observation from~\cite{Gal19} about sums of binomial coefficients.

\begin{lemma}[\cite{Gal19}]\label{lem: sums of bincoeffs}
Let $b(n)$ be a function with $b(n) = o(n)$. Then there is $n_0 \in \NN$ such that for every $n \geq n_0$,
\[
	\sum_{j = 0}^{b(n)} \binom{n}{j}
	\leq 2^{3b(n)\log \frac n{b(n)} }
	\,.
\]
\end{lemma}

Using this, we can bound $\abs {\II(G)}$ in terms of the partition function $\Xi_{Z,b(n)}$.

\begin{lemma}\label{lem: split into polymer models}
Let $k \geq 3$, $t \in \NN$, and $\beta >0$. Furthermore, let $b(n)$ be a function with $b(n) \geq \frac {\beta t n}{\log_{\gamma_k} n}$ and $b(n) = o(n)$. Then there is $n_0 \in \NN$ such that the following holds for every $n \geq n_0$:

Let $G$ be an $r$-regular $k$-partite $k$-graph with vertex partition $\mathcal Z$, where each part has order $n$. If $G$ satisfies \Reg{t}, \ExpL{\beta}, and \Def{b(n)}, then
\begin{align*}
	2^{(k-1)n} \sum_{Z \in \mathcal Z} \Xi_{Z,b(n)} - k^3 \cdot 2^{(k-\frac 32)n}
	\leq \abs {\II(G)}
	\leq 2^{(k-1)n} \sum_{Z \in \mathcal Z} \Xi_{Z,b(n)}
	\,.
\end{align*}
\end{lemma}

\begin{proof}
For $Z \in \mathcal Z$ and $I \in \II(G)$, we can regard $I \cap Z$ as a vertex subset of $Z_2$. We say that~$Z$ is a \emph{defect class} for $I$ if every connected component of $Z_2[I \cap Z]$ has order at most $b(n)$. In particular, \Def{b(n)} guarantees that every $I \in \II(G)$ has at least one defect class, so
\begin{align}\label{eq: count indepsets by defect classes}
	\abs {\II(G)}
	\leq \sum_{Z \in \mathcal Z} \abs { \{ I \in \II(G) \colon \text{$Z$ is a defect class for $I$} \} }
	\,.
\end{align}
The following two things remain to show: Firstly, we argue that the number of $I \in \II(G)$ with more than one defect class is at most $k^2 \cdot 2^{(k-\frac 32)n}$. Since every $I$ has at most $k$ defect classes, this implies that (\ref{eq: count indepsets by defect classes}) overcounts $\abs {\II(G)}$ by at most $k^3 \cdot 2^{(k - \frac 32)n}$. Secondly, we prove that for every $Z \in \mathcal Z$,
\begin{align}\label{eq: relationship indepset polymers}
	\abs { \{ I \in \II(G) \colon \text{$Z$ is a defect class for $I$} \} }
	= 2^{(k-1)n} \Xi_{Z,b(n)}
	\,.
\end{align}
This then implies the assertion.

In order to limit the number of $I \in \II(G)$ with multiple defect classes, we first establish the following observation.

\begin{claim}\label{claim: intersection with defect class}
If $Z \in \mathcal Z$ is a defect class for $I \in \II(G)$, then $\abs {I \cap Z} \leq \frac k \beta b(n)$.
\end{claim}

\begin{proofclaim}
First note that $b(n) \geq \frac {\beta tn}{\log_{\gamma_k} n}$ and \Reg{t} guarantee that
\begin{align}\label{eq: b/n beats 1/r}
	\frac {\beta n}{b(n)}
	\leq \frac 1t \log_{\gamma_k} n
	\leq r
	\,.
\end{align}
Next, consider the vertex set $S$ of an arbitrary connected component of $Z_2[I \cap Z]$, where $Z$ is a defect class for $I$. Then $\abs S \leq b(n)$ and thus by \cref{lem: ExpL relative expansion} and (\ref{eq: b/n beats 1/r}), $\abs {N_G(S)} \geq \frac {\beta n}{b(n)} \abs S$. Now let $C_2(I \cap Z)$ be the set consisting of the vertex sets of all connected components of $Z_2[I \cap Z]$. Since these components are pairwise disconnected in $Z_2[I \cap Z]$, the neighbourhoods $N_G(S)$ of the $S \in C_2(I \cap Z)$ must be pairwise disjoint. We obtain
\begin{align*}
	\abs {N_G(I \cap Z)}
	= \sum_{S \in C_2(I \cap Z)} \abs {N_G(S)}
	\geq \frac {\beta n}{b(n)} \sum_{S \in C_2(I \cap Z)} \abs S
	= \frac {\beta n}{b(n)} \abs {I \cap Z}
	\,.
\end{align*}
However, $\abs {N_G(I \cap Z)} \leq \abs {V(G)} = kn$, so $\abs {I \cap Z} \leq \frac k \beta b(n)$ must hold. This proves the claim.
\end{proofclaim}

This claim allows us to bound the number of $I \in \II(G)$ with multiple defect classes from above by the number of sets $J \subset V(G)$ that intersect at least two distinct $Z, Z' \in \mathcal Z$ in $\abs {J \cap Z}, \abs {J \cap Z'} \leq \frac k \beta b(n)$ vertices. This can be done by a straightforward counting argument: Firstly, there are $\binom{k}{2}$ possibilities to select two distinct partition classes $Z, Z' \in \mathcal Z$. Next, there are $\sum_{j = 0}^{\frac k \beta b(n)} \binom{n}{j}$ possibilities to choose $J \cap Z$ and $J \cap Z'$, respectively. Finally, $J$ is completely specified by picking any of the $2^{(k-2)n}$ subsets of $\overline{Z \cup Z'}$ as $J \cap \overline{Z \cup Z'}$. The number of independent sets with multiple defect classes is therefore at most
\[
	\binom{k}{2} \left( \sum_{j = 0}^{\frac k \beta b(n)} \binom{n}{j} \right)^2 2^{(k-2)n}
	\,.
\]
According to \cref{lem: sums of bincoeffs}, $b'(n) \coloneqq \frac k \beta b(n) = o(n)$ implies that
\[
	\sum_{j = 0}^{b'(n)} \binom{n}{j}
	\leq 2^{3b'(n) \log \frac n{b'(n)}}
	= 2^{n \left( 3 \frac {b'(n)}n \log \frac n{b'(n)} \right)}
	\,.
\]
Moreover, $b'(n) = o(n)$ guarantees that $\frac n{b'(n)}$ becomes arbitrarily large with increasing $n$, so $3 \frac {b'(n)}n \log \frac n{b'(n)} \leq \frac 14$ for sufficiently large $n$. It follows that the number of $I \in \II(G)$ with multiple defect classes is bounded from above by
\[
	\binom{k}{2} \left( \sum_{j = 0}^{b'(n)} \binom{n}{j} \right)^2 2^{(k-2)n}
	\leq k^2 \cdot \left( 2^{\frac n4} \right)^2 \cdot 2^{(k-2)n}
	= k^2 \cdot 2^{(k- \frac 32)n}
	\,.
\]
	
It remains to establish (\ref{eq: relationship indepset polymers}). We begin by fixing an arbitrary partition class $Z \in \mathcal Z$ and set $\II_Z(G) \coloneqq \{ I \in \II(G) \colon \text{$Z$ is a defect class for $I$} \}$. We call a set $T \subset Z$ a \emph{defect set} if every connected component of $Z_2[T]$ has order at most $b(n)$. This means that $Z$ is a defect class for $I$ if and only if $I \cap Z$ is a defect set. For $T \subset Z$, we now write $\II_Z^T(G) \coloneqq \{ I \in \II_Z(G) \colon I \cap Z = T \}$ and conclude that
\begin{align}\label{eq: split indepsets by intersection with defect class}
	\abs {\II_Z(G)} 
	= \sum_{\substack{T \subset Z\\\text{$T$ defect set}}} \abs {\II_Z^T(G)}
	\,.
\end{align}

In order to examine $\abs {\II_Z^T(G)}$, we observe that when completing a given defect set $T \subset Z$ to a set $I \in \II_Z^T(G)$, a vertex $v \in \overline Z$ has to be treated one of two ways:
\begin{itemize}
	\item If $v \notin N_G(T)$, we are free to select $v$ independently from all other choices we make because for all edges $e \in E(G)$ that contain $v$, the vertex in $e \cap Z$ does not belong to $T$ and thus not to $I$.
	\item If $v \in N_G(T)$, selecting $v$ prevents us from selecting any set of $k-2$ vertices in $N_G(T)$ that forms an edge with $v$ and any vertex in $T$. This means that the intersection of $I$ with $N_G(T)$ has to be an independent set in the $(k-1)$-graph $L_G(T)$.
\end{itemize}
This shows that there are exactly 
\begin{align}\label{eq: complete intersection with defect class to indepset}
	\abs {\II_Z^T(G)} 
	= \abs {\II(L_G(T))} \cdot 2^{\abs {\overline Z \setminus N_G(T)}} 
	= \abs {\II(L_G(T))} \cdot 2^{(k-1)n - \abs {N_G(T)}}
\end{align}
ways to complete a given defect set $T \subset Z$ to $I \in \II_Z^T(G)$. We find that
\begin{align} \label{eq: partition II(G)}
	\abs {\II_Z(G)}
	\overset{(\ref{eq: split indepsets by intersection with defect class})}= \sum_{\substack{T \subset Z\\\text{$T$ defect set}}} \abs {\II_Z^T(G)}
	\overset{(\ref{eq: complete intersection with defect class to indepset})}= 2^{(k-1)n} \sum_{\substack{T \subset Z\\\text{$T$ defect set}}} \abs {\II(L_G(T))} \cdot 2^{-\abs {N_G(T)}}
	\,.
\end{align}

Next, consider the defect set $T \subset Z$ as a vertex subset of $Z_2$ and let $C_2(T)$ be the set consisting of the vertex sets of all connected components of $Z_2[T]$. This means that $N_G(T)$ is the disjoint union of all $N_G(S)$ with $S \in C_2(T)$ and consequently, $L_G(T)$ is the disjoint union of all $L_G(S)$ with $S \in C_2(T)$. In particular, we obtain
\[
	\abs {\II(L_G(T))} \cdot 2^{-\abs {N_G(T)}} 
	= \prod_{S \in C_2(T)} \abs {\II(L_G(S))} \cdot 2^{-\abs {N_G(S)}}
	\,.
\]
	
Note that by construction, every $S \in C_2(T)$ is $2$-linked and non-empty. Since $T$ is a defect set, we moreover have $\abs S \leq b(n)$ for all $S \in C_2(T)$ and conclude that each $S \in C_2(T)$ is a polymer in $\PM_{Z,b(n)}$. We recall that 
\begin{align*}\tag{\ref{eq: polymer weights}}
	w(S) 
	= \abs {\II(L_G(S))} \cdot 2^{-\abs {N_G(S)}}
	\,.
\end{align*}
Again by construction, the set $C_2(T)$ of polymers is compatible. Instead of summing over all defect sets $T \subset Z$, we can therefore sum over all compatible $\mathcal S \subset \PM_{Z,b(n)}$. Conversely, every compatible $\mathcal S \subset \PM_{Z,b(n)}$ corresponds to the defect set $T \coloneqq \bigcup_{S \in \mathcal S} S \subset Z$. This shows that
\begin{align}\label{eq: from defect sets to partition function}
	\sum_{\substack{T \subset Z\\\text{$T$ defect set}}} \abs {\II(L_G(T))} \cdot 2^{-\abs {N_G(T)}}
	= \sum_{\substack{\mathcal S \subset \PM_Z\\\text{$\mathcal S$ compatible}}} \prod_{S \in \mathcal S} w(S)
	\overset{(\ref{eq: partition function})}= \Xi_{Z,b(n)}
	\,.
\end{align}
In total, we obtain
\[
	\abs {\II_Z(G)}
	\overset{(\ref{eq: partition II(G)})}= 2^{(k-1)n} \sum_{\substack{T \subset Z\\\text{$T$ defect set}}} \abs {\II(L_G(T))} \cdot 2^{-\abs {N_G(T)}}
	\overset{(\ref{eq: from defect sets to partition function})}= 2^{(k-1)n} \Xi_{Z,b(n)}
	\,,
\]
which establishes (\ref{eq: relationship indepset polymers}) and thus finishes the proof.
\end{proof}

The main technical ingredient in our proof is verifying the Kotecký-Preiss condition (see \cref{thm: KP}) for this polymer model. We will choose the functions $f, g$ for $S \in \PM_{Z,b(n)}$ as follows.
\begin{align}\label{eq: g}\notag
	f(S) 
	&\coloneqq \frac {(k-1)\abs S}r \\
	g(S)
	&\coloneqq \log \gamma_k \cdot r \log (2 \abs S)
	\,.
\end{align}
Note that $\gamma_k = \frac {2^{k-1}}{2^{k-1}-1} > 1$ and polymers $S \in \PM_{Z,b(n)}$ are non-empty, so $g$ is strictly positive.

\begin{lemma}\label{lem: KP condition holds}
Let $k \geq 3$, $t \in \NN$, and $\beta > 0$. Furthermore, let $b(n)$ be a function with $b(n) = o(n)$. Then there is $n_0 \in \NN$ such that the following holds for every $n \geq n_0$:

Let $G$ be an $r$-regular $k$-partite $k$-graph with vertex partition $\mathcal Z$, where each part has order $n$. If $G$ satisfies \Reg{t}, \ExpM{\alpha_{k,t}}, and \ExpL{\beta}, then for every $u \in Z \in \mathcal Z$,
\[
	\sum_{\substack{S \in \PM_{Z,b(n)} \\ S \ni u}} w(S) \exp \big( f(S) + g(S) \big)
	= \sum_{\substack{S \in \PM_{Z,b(n)} \\ S \ni u}} w(S) \exp \left( \frac {(k-1) \abs S}r + \log \gamma_k \cdot r \log (2 \abs S) \right)
	\leq \frac 1{r^3}
	\,.
\]
\end{lemma}

Proving \cref{lem: KP condition holds} is the main part of the paper and deferred to~\cref{sec: KP}. The proof splits into \cref{lem: S polynomial controllable,lem: below delta n/r,lem: above delta n/r}. For the remainder of the current section, we focus on how to deduce our main theorem from \cref{lem: KP condition holds} via suitable truncation of the cluster expansion.

\begin{proof}[Proof of \cref{thm: main}]
Given $\rho > 0$, let $\rho' > 0$ be chosen sufficiently small such that $e^{\rho'} \leq 1 + \rho$. Also assume without loss of generality that $b(n) \geq \frac {\beta tn}{\log_{\gamma_k} n} \geq t$ for sufficiently large $n$. Note that any polymer $S$ contained in a cluster $\Gamma \in \CM_{Z,b(n)}$ of size $\norm \Gamma \leq t$ automatically satisfies $\abs S \leq \norm \Gamma \leq t$, so the right-hand side of the formula in \cref{thm: main} does not change if we replace $\CM_{Z,t}$ by $\CM_{Z,b(n)}$. We show that for any partition class $Z \in \mathcal Z$, we have 
\begin{align}\label{eq: partition function via cluster expansion trunc}
	\log \Xi_{Z,b(n)}
	= \sum_{m = 1}^t \sum_{\substack{\Gamma \in \CM_{Z,b(n)} \\ \norm \Gamma = m}} w(\Gamma) \pm \rho' n \gamma_k^{-rt}
	\,.
\end{align}
Plugging this into \cref{lem: split into polymer models}, we obtain an upper bound of
\[
	\abs {\II(G)}
	\overset{L\ref{lem: split into polymer models}}\leq 2^{(k-1)n} \sum_{Z \in \mathcal Z} \Xi_{Z,b(n)}
	\leq 2^{(k-1)n} \cdot \sum_{Z \in \mathcal Z} \exp \left( \sum_{m = 1}^t \sum_{\substack{\Gamma \in \CM_{Z,b(n)} \\ \norm \Gamma = m}} w(\Gamma) + \rho' n \gamma_k^{-rt} \right)
	\,.
\]
Since $r \geq \frac 1t \log_{\gamma_k} n$ by \Reg{t}, the final term accounts for a multiplicative error of at most $e^{\rho'} \leq 1 + \rho$, as desired.

By considering only subsets of $G$ that have empty intersection with a fixed partition class, it is also easy to see that $\abs {\II(G)} \geq 2^{(k-1)n}$, so by the upper bound of \cref{lem: split into polymer models}, we conclude that $\sum_{Z \in \mathcal Z} \Xi_{Z,b(n)} \geq \abs {\II(G)} \cdot 2^{-(k-1)n} \geq 1$. In particular, the ($k^3 \cdot 2^{(k-\frac 32)n}$)-term in the lower bound of \cref{lem: split into polymer models} is negligible in comparison to $2^{(k-1)n} \sum_{Z \in \mathcal Z} \Xi_{Z,b(n)}$ and we obtain the desired lower bound of
\begin{align*}
	\abs {\II(G)}
	&\overset{L\ref{lem: split into polymer models}}\geq 2^{(k-1)n} \sum_{Z \in \mathcal Z} \Xi_{Z,b(n)} - k^3 \cdot 2^{(k-\frac 32)n} \\
	&\overset{\hphantom{L\ref{lem: split into polymer models}}}\geq \left( 2^{(k-1)n} - k^3 \cdot 2^{(k-\frac 32)n} \right) \sum_{Z \in \mathcal Z} \Xi_{Z,b(n)}
	\geq (1 - \rho) \cdot 2^{(k-1)n} \sum_{Z \in \mathcal Z} \Xi_{Z,b(n)}
	\,.
\end{align*}

In order to establish (\ref{eq: partition function via cluster expansion trunc}), fix $Z \in \mathcal Z$ and $S \in \PM_{Z,b(n)}$ arbitrarily. By definition, a polymer $T \in \PM_{Z,b(n)}$ is incompatible to $S$ if and only if $T$ contains a vertex $u \in S \cup N_{Z_2}(S)$. Since every $v \in S$ has at most $(k-1)r$ neighbours $x \in \overline Z$ and every such $x$ has at most $r$ neighbours in $Z$, we find that $\abs {S \cup N_{Z_2}(S)} \leq (k-1)r^2 \abs S$. Combining these observations with \cref{lem: KP condition holds}, we obtain
\begin{align*}
	\sum_{\substack{T \in \PM_{Z,b(n)} \\ T \nsim S}} w(T) \exp \big( f(T) + g(T) \big)
	&\overset{\hphantom{L\ref{lem: KP condition holds}}}\leq \sum_{u \in S \cup N_{Z_2}(S)} \sum_{\substack{T \in \PM_{Z,b(n)} \\ T \ni u}} w(T) \exp \big( f(T) + g(T) \big) \\
	&\overset{L\ref{lem: KP condition holds}}\leq \frac {\abs {S \cup N_{Z_2}(S)}}{r^3} \\
	&\overset{\hphantom{L\ref{lem: KP condition holds}}}\leq \frac {(k-1) \abs S}{r}
	= f(S)
	\,,
\end{align*}
which enables the application of \cref{thm: KP}. Therefore, the cluster expansion of $(\PM_{Z,b(n)}, \sim, w)$ converges absolutely and by (\ref{eq: cluster expansion}), we have
\begin{align}\label{eq: partition function via cluster expansion full}
	\log \Xi_{Z,b(n)}
	= \sum_{m = 1}^\infty \sum_{\substack{\Gamma \in \CM_{Z,b(n)} \\ \norm \Gamma = m}} w(\Gamma)
	\,.
\end{align}
	
We now fix an arbitrary $m \leq e^t$ and examine the behaviour of $\sum_{\Gamma \in \CM_{Z,b(n)}, \norm \Gamma = m} \abs {w(\Gamma)}$ as $n \to \infty$. Recall that by (\ref{eq: polymer}), polymers are non-empty, so $\norm \Gamma = m$ implies $\abs \Gamma \leq m$. Since there are only finitely many graphs with at most $m$ vertices, there is a constant $c_m$ not depending on $n,r$ such that $\abs {\phi(H_\Gamma)} \leq c_m$ for every $\Gamma \in \CM_{Z,b(n)}$ with $\norm \Gamma = m$.

According to (\ref{eq: cluster weights}) and (\ref{eq: polymer weights}), it remains to consider the term
\[
	\prod_{S \in \Gamma} \abs {\II(L_G(S))} \cdot 2^{-\abs {N(S)}}
	\,.
\]
Note that each polymer $S$ has order at most $\abs S \leq \norm \Gamma = m \leq e^t$ and satisfies 
\begin{align}\label{eq: indepsets in prefix}
	\abs {\II(L_G(S))} 
	\leq (2^{k-1}-1)^{\abs {E(L_G(S))}} 
	\leq (2^{k-1}-1)^{r \abs S}	
\end{align}
by regularity. Choosing $n$ and thus $r \geq \frac 1t \log_{\gamma_k} n$ by \Reg{t} sufficiently large, we can ensure $e^t \leq r$, which allows us to apply \ExpM{\alpha_{k,t}} to $S$. Moreover, we recall $\alpha_{k,t} \leq \frac {\log_2 \gamma_k}{2e^{2t}}$ by (\ref{eq: alpha0}) to observe that
\begin{align}\label{eq: neighbours in prefix}
	2^{-\abs {N(S)}}
	\overset{\mathrm{Exp_1}}\leq 2^{-(k-1-\alpha_{k,t}) r \abs S}
	\leq 2^{\alpha_{k,t} e^t r} \cdot 2^{-(k-1) r \abs S}
	\overset{(\ref{eq: alpha0})}\leq \gamma_k^{\frac r{2 e^t}} \cdot 2^{-(k-1) r \abs S}
	\,.
\end{align}

Together with $\abs \Gamma \leq \norm \Gamma \leq e^t$, we can thus bound the weight of any cluster $\Gamma \in \CM_{Z,b(n)}$ with $\norm \Gamma = m$ by
\begin{align}\label{eq: weights in prefix}
	\notag \abs {w(\Gamma)}
	&\overset{\substack{(\ref{eq: cluster weights}) \\ (\ref{eq: polymer weights})}}{\underset{\hphantom{\substack{(\ref{eq: indepsets in prefix}) \\ (\ref{eq: neighbours in prefix})}}}\leq} \abs {\phi(H_\Gamma)} \prod_{S \in \Gamma} \abs {\II(L_G(S))} \cdot 2^{-\abs {N(S)}} \\
	\notag &\overset{\substack{(\ref{eq: indepsets in prefix}) \\ (\ref{eq: neighbours in prefix})}}\leq c_m \gamma_k^{\frac {r \abs \Gamma}{2 e^t}} \prod_{S \in \Gamma} \left( \frac {2^{k-1} - 1}{2^{k-1}} \right)^{r \abs S} \\
	&\overset{\hphantom{\substack{(\ref{eq: indepsets in prefix}) \\ (\ref{eq: neighbours in prefix})}}}\leq c_m \gamma_k^{\frac r2 - rm}
	\,.
\end{align}

It remains to bound the number of such clusters. For this, we first observe that by connectedness of $H_\Gamma$, the set $V(\Gamma) \coloneqq \bigcup_{S \in \Gamma} S$ must be $2$-linked. Moreover, it contains at most $\norm \Gamma = m$ vertices. We can specify all such sets by first picking a number $\ell \leq m$, then a vertex $u \in Z$, and then one of the according to \cref{lem: 2-linked fixed size} at most $e((k-1)er^2)^{\ell-1} \leq e((k-1)er^2)^{m-1}$ $2$-linked subsets of $Z$ of order $\ell$ that contain $u$. This shows that there are at most $emn((k-1)er^2)^{m-1}$ candidates for $V(\Gamma)$. Fixing one of these candidates, we pick another number $\ell' \leq m$, as well as $\ell'$ arbitrary subsets of $V(\Gamma)$ to be the polymers in $\Gamma$. Since $\abs {V(\Gamma)} \leq m$, we have $m(2^m)^{\ell'} \leq m \cdot 2^{m^2}$ possibilities in this step. In total, we find that there is a constant $c'_m$ not depending on $n,r$ such that
\begin{align}\label{eq: clusters in prefix}
	\abs {\{ \Gamma \in \CM_{Z,b(n)} \colon \norm \Gamma = m \}} 
	\leq c'_mnr^{2(m-1)}
	\,.
\end{align}

Setting $C_m \coloneqq c_m c'_m$ and combining (\ref{eq: weights in prefix}) with (\ref{eq: clusters in prefix}), we observe that for every $m \leq e^t$, we have
\begin{align}\label{eq: contribution prefix}
	\sum_{\substack{\Gamma \in \CM_{Z,b(n)} \\ \norm \Gamma = m}} \abs {w(\Gamma)}
	\leq C_m nr^{2(m-1)} \gamma_k^{\frac r2 - rm}
	\,.
\end{align}
Now let $C \coloneqq e^t \max \{ C_m \colon t+1 \leq m \leq e^t\}$ and note that $n$ and thus also $r \geq \frac 1t \log_{\gamma_k} n$ can be chosen sufficiently large. Since $\gamma_k > 1$, we can therefore guarantee that $r^{2(m-1)} \leq r^{2(e^t-1)} \leq \frac {\rho'}{2C} \gamma_k^{\frac r2}$ for every $m \leq e^t$. In particular, we observe that
\begin{align}\label{eq: finite prefix of sum negligible}
	\sum_{m = t+1}^{e^t} \sum_{\substack{\Gamma \in \CM_{Z,b(n)} \\ \norm \Gamma = m}} \abs {w(\Gamma)}
	\overset {(\ref{eq: contribution prefix})}\leq n\gamma_k^{\frac r2} \sum_{m = t+1}^{e^t} C_m r^{2(m-1)} \gamma_k^{-rm}
	\leq \frac {\rho'}{2C} n\gamma_k^r \sum_{m = t+1}^{e^t} C_m \gamma_k^{-r(t+1)}
	\leq \frac {\rho'} 2 n \gamma_k^{-rt}
	\,.
\end{align}
	
Finally, we direct our attention to clusters $\Gamma \in \CM_{Z,b(n)}$ with $\norm \Gamma > e^t$. Since every cluster contains some vertex $v$ and is in particular incompatible to the polymer $\{ v \}$, we can easily bound
\begin{align}\label{eq: tail setup}
	\sum_{\substack{\Gamma \in \CM_{Z,b(n)} \\ \norm \Gamma > e^t}} \abs {w(\Gamma)}
	\leq \sum_{v \in Z} \sum_{\substack{\Gamma \in \CM_{Z,b(n)} \\ \norm \Gamma > e^t \\\Gamma \ni v}} \abs {w(\Gamma)}
	\leq \sum_{v \in Z} \sum_{\substack{\Gamma \in \CM_{Z,b(n)} \\ \norm \Gamma > e^t \\\Gamma \nsim \{ v \}}} \abs {w(\Gamma)}
	\,.
\end{align}
	
In order to show that this contribution is also negligible, we use the tail bound provided by \cref{thm: KP} with $S \coloneqq \{ v \}$. Again choosing $n$ and thus $r \geq \frac 1t \log_{\gamma_k} n$ sufficiently large, this guarantees that
\begin{align}\label{eq: tail from KP}
	\sum_{v \in Z} \sum_{\substack{\Gamma \in \CM_{Z,b(n)} \\ \norm \Gamma > e^t \\\Gamma \nsim \{ v \}}} \abs {w(\Gamma)} \exp \left( \sum_{T \in \Gamma} g(T) \right)
	\leq \abs Z f(\{ v \})
	= n \cdot \frac {k-1}{r}
	\leq \frac {\rho'} 2 n
	\,,
\end{align}
where $g(T) = \log \gamma_k \cdot r \log (2 \abs T)$ by (\ref{eq: g}). Before we can plug (\ref{eq: tail from KP}) into (\ref{eq: tail setup}), we have to examine the exponential term. We observe that since polymers $T \in \Gamma$ are non-empty by (\ref{eq: polymer}), we have
\begin{align}\label{eq: estimate log in g}
	\sum_{T \in \Gamma} \log (2 \abs T)
	= \log \left( \prod_{T \in \Gamma} 2 \abs T \right)
	\geq \log \left( 2^{\abs \Gamma} \max_{T \in \Gamma} \abs T \right)
	\geq \log \left( \abs \Gamma \max_{T \in \Gamma} \abs T \right)
	\geq \log \norm \Gamma
	\,.
\end{align}
Since $\norm \Gamma > e^t$, this implies
\begin{align}\label{eq: estimate g-term in KP output}
	\exp \left( \sum_{T \in \Gamma} g(T) \right)
	\overset{(\ref{eq: g})}= \gamma_k^{r \sum_{T \in \Gamma} \log (2 \abs T)}
	\overset{(\ref{eq: estimate log in g})}\geq \gamma_k^{r \log \norm \Gamma}
	\geq \gamma_k^{rt}
	\,.
\end{align}
Using (\ref{eq: tail from KP}), we obtain
\begin{align}\label{eq: tail from KP modified}
	\gamma_k^{rt} \sum_{v \in Z} \sum_{\substack{\Gamma \in \CM_{Z,b(n)} \\ \norm \Gamma > e^t \\\Gamma \nsim \{ v \}}} \abs {w(\Gamma)}
	\overset{(\ref{eq: estimate g-term in KP output})}\leq \sum_{v \in Z} \sum_{\substack{\Gamma \in \CM_{Z,b(n)} \\ \norm \Gamma > e^t \\\Gamma \nsim \{ v \}}} \abs {w(\Gamma)} \exp \left( \sum_{T \in \Gamma} g(T) \right)
	\overset{(\ref{eq: tail from KP})}\leq \frac {\rho'} 2 n
	\,.
\end{align}
After division by $\gamma_k^{rt}$, this can be plugged into (\ref{eq: tail setup}) to yield
\begin{align}\label{eq: tail negligible}
	\sum_{\substack{\Gamma \in \CM_{Z,b(n)} \\ \norm \Gamma > e^t}} \abs {w(\Gamma)}
	\overset{(\ref{eq: tail setup})}\leq \sum_{v \in Z} \sum_{\substack{\Gamma \in \CM_{Z,b(n)} \\ \norm \Gamma > e^t \\\Gamma \nsim \{ v \}}} \abs {w(\Gamma)}
	\overset{(\ref{eq: tail from KP modified})}\leq \frac {\rho'} 2 n \gamma_k^{-rt}
	\,.
\end{align}
	
Combining (\ref{eq: partition function via cluster expansion full}), (\ref{eq: finite prefix of sum negligible}), and (\ref{eq: tail negligible}), we obtain
\begin{align*}
	\log \Xi_{Z,b(n)}
	\overset{(\ref{eq: partition function via cluster expansion full})}{\underset{\hphantom{\substack{(\ref{eq: finite prefix of sum negligible}) \\ (\ref{eq: tail negligible})}}}=} &\sum_{m = 1}^t \sum_{\substack{\Gamma \in \CM_{Z,b(n)} \\ \norm \Gamma = m}} w(\Gamma) \pm \sum_{m = t+1}^{e^t} \sum_{\substack{\Gamma \in \CM_{Z,b(n)} \\ \norm \Gamma = m}} \abs {w(\Gamma)} \pm \sum_{m > e^t} \sum_{\substack{\Gamma \in \CM_{Z,b(n)} \\ \norm \Gamma = m}} \abs {w(\Gamma)} \\
	\overset{\substack{(\ref{eq: finite prefix of sum negligible}) \\ (\ref{eq: tail negligible})}}= &\sum_{m = 1}^t \sum_{\substack{\Gamma \in \CM_{Z,b(n)} \\ \norm \Gamma = m}} w(\Gamma) \pm \rho' n \gamma_k^{-rt}
	\,.
\end{align*}
This establishes (\ref{eq: partition function via cluster expansion trunc}) and, as initially explained, finishes the proof.
\end{proof}

\section{Verification of the Kotecký-Preiss condition}\label{sec: KP}
We still have to prove \cref{lem: KP condition holds}, which we split into \cref{lem: S polynomial controllable,lem: below delta n/r,lem: above delta n/r}. The first one only considers polymers $S \in \PM_{Z,b(n)}$ with $\abs S \leq r$. In this regime, their $2$-linkedness and strong expansion by \ExpM{\alpha_{k,t}} are sufficient to limit their contribution.

\begin{lemma}\label{lem: S polynomial controllable}
Let $k \geq 3$. Furthermore, let $b(n)$ be a function with $b(n) = o(n)$. Then there is $n_0 \in \NN$ such that the following holds for every $n \geq n_0$:

Let $G$ be an $r$-regular $k$-partite $k$-graph with vertex partition $\mathcal Z$, where each part has order $n$. If $G$ satisfies \Reg{t} and \ExpM{\alpha_{k,t}}, then for every $u \in Z \in \mathcal Z$,
\[
	\sum_{s = 1}^r \sum_{\substack{S \in \PM_{Z,b(n)} \\ S \ni u \\ \abs S = s}} w(S) \exp \left( \frac {(k-1) s}r + \log \gamma_k \cdot r \log (2s) \right)
	\leq \frac 1{3r^3}
	\,.
\]
\end{lemma}

\begin{proof}
Let $1 \leq s \leq r$ and $S \in \PM_{Z,b(n)}$ with $\abs S = s$ be arbitrary. Recall that (\ref{eq: polymer weights}) defines polymer weights as $w(S) = \abs {\II(L_G(S))} \cdot 2^{-\abs {N(S)}}$. Since $G$ is $r$-regular, the first factor of this can be trivially bounded by
\begin{align}\label{eq: indepset link graph trivial bound}
	\abs {\II(L_G(S))}
	\leq (2^{k-1} - 1)^{\abs {E(L_G(S))}}
	\leq (2^{k-1} - 1)^{rs}
	\,.
\end{align}
Setting $h \coloneqq \abs {N(S)}$, we note that \ExpM{\alpha_{k,t}} guarantees that $rs \leq \frac {h}{k-1-\alpha_{k,t}}$. In total, we obtain
\begin{align}\label{eq: polymer weight trivial bound}
	w(S)
	\overset{\substack{(\ref{eq: polymer weights}) \\ (\ref{eq: indepset link graph trivial bound})}}\leq (2^{k-1} - 1)^{rs} \cdot 2^{-h}
	\overset{\mathrm{Exp_1}}\leq \exp \left( h \cdot \left( \frac {\log (2^{k-1} - 1)}{k-1-\alpha_{k,t}} - \log 2 \right) \right)
	\,.
\end{align}
In order to bound the exponential term in the statement of \cref{lem: S polynomial controllable}, we observe that $s \leq r$ by assumption and $\log (2s) \leq s \log 2$ holds for any $s \in \NN$. Also, we again use that $rs \leq \frac h{k-1-\alpha_{k,t}}$ by \ExpM{\alpha_{k,t}}. Altogether, we obtain
\begin{align}\label{eq: exponential terms bound}
	\exp \left( \frac {(k-1)s}r + \log \gamma_k \cdot r \log(2s) \right)
	\overset{\mathrm{Exp_1}}\leq e^{k-1} \exp \left( h \cdot \frac {\log 2 \cdot \log \gamma_k}{k-1-\alpha_{k,t}} \right)
	\,.
\end{align}
We now split the sum in the statement of \cref{lem: S polynomial controllable} according to $h = \abs {N(S)} \geq r$ and apply (\ref{eq: polymer weight trivial bound}) as well as (\ref{eq: exponential terms bound}) to obtain
\begin{align}\label{eq: weights and exponential terms bound}
	\notag &\sum_{\substack{S \in \PM_{Z,b(n)} \\ S \ni u \\ \abs S = s}} w(S) \exp \left( \frac {(k-1) s}r + \log \gamma_k \cdot r \log (2s) \right) \\
	&\qquad \overset{\substack{(\ref{eq: polymer weight trivial bound}) \\ (\ref{eq: exponential terms bound})}}\leq e^{k-1} \sum_{h = r}^\infty \sum_{\substack{S \in \PM_{Z,b(n)} \\ S \ni u \\ \abs S = s \\ \abs {N(S)} = h}} \exp \left( h \cdot \left( \frac {\log (2^{k-1} - 1) + \log 2 \cdot \log \gamma_k}{k-1-\alpha_{k,t}} - \log 2 \right) \right)
	\,.
\end{align}

Next, we bound the number of $S \in \PM_{Z,b(n)}$ with $S \ni u$ and $\abs S = s$. Since we are only interested in an upper bound, we may ignore the additional restriction of $\abs {N(S)} = h$. According to \cref{lem: 2-linked fixed size}, there are at most $e((k-1)er^2)^{s-1}$ such polymers $S$. Choosing $n$ sufficiently large and recalling that $r \geq \frac 1t \log_{\gamma_k} n$ holds by \Reg{t}, we observe that
\begin{align}\label{eq: number of polymers}
	\abs {\{ S \in \PM_{Z,b(n)} \colon \text{$S \ni u$ and $\abs S = s$} \}}
	\leq e((k-1)er^2)^{s-1}
	\leq r^{3s}
	= \exp \left( 3s \log r \right)
	\,.
\end{align}
Applying $s \leq \frac h{(k-1-\alpha_{k,t})r}$ (guaranteed by \ExpM{\alpha_{k,t}}) to this leads to
\begin{align}\label{eq: inner sum}
	\notag &\sum_{\substack{S \in \PM_{Z,b(n)} \\ S \ni u \\ \abs S = s \\ \abs {N(S)} = h}} \exp \left( h \cdot \left( \frac {\log (2^{k-1} - 1) + \log 2 \cdot \log \gamma_k}{k-1-\alpha_{k,t}} - \log 2 \right) \right) \\
	&\qquad \overset{\substack{(\ref{eq: number of polymers}) \\ \mathrm{Exp_1}}}\leq \exp \left( h \cdot \left( \frac {3 \log r}{(k-1-\alpha_{k,t})r} + \frac {\log (2^{k-1} - 1) + \log 2 \cdot \log \gamma_k}{k-1-\alpha_{k,t}} - \log 2 \right) \right)
	\,.
\end{align}

We now note that by choosing $n$ and thus $r$ sufficiently large, the $\frac {\log r}r$-term in (\ref{eq: inner sum}) becomes negligible. Since $\alpha_{k,t} < \frac {(k-1)(1 - \log 2) \log \gamma_k} {\log (2^{k-1} - 1) + \log \gamma_k}$ by (\ref{eq: alpha0}), we obtain
\begin{align}\label{eq: exponent is negative}
	\frac {3 \log r}{(k-1-\alpha_{k,t})r} + \frac {\log (2^{k-1} - 1) + \log 2 \cdot \log \gamma_k}{k-1-\alpha_{k,t}} - \log 2
	< \frac {\log (2^{k-1} - 1) + \log 2 \cdot \log \gamma_k}{k-1-\frac {(k-1)(1 - \log 2) \log \gamma_k} {\log (2^{k-1} - 1) + \log \gamma_k}} - \log 2
	\,.
\end{align}
Using $\gamma_k = \frac {2^{k-1}}{2^{k-1}-1}$, a short calculation shows that the right-hand side of (\ref{eq: exponent is negative}) evaluates to $0$. Since (\ref{eq: exponent is negative}) is a strict inequality, we conclude that there must be some $q > 0$ not depending on $n,r$ such that
\begin{align*}
	\notag &\sum_{s = 1}^r \sum_{\substack{S \in \PM_{Z,b(n)} \\ S \ni u \\ \abs S = s}} w(S) \exp \left( \frac {(k-1) s}r + \log \gamma_k \cdot r \log (2s) \right) \\
	&\qquad \overset{(\ref{eq: weights and exponential terms bound})}\leq e^{k-1} \sum_{s = 1}^r \sum_{h = r}^\infty \sum_{\substack{S \in \PM_{Z,b(n)} \\ S \ni u \\ \abs S = s \\ \abs {N(S)} = h}} \exp \left( h \cdot \left( \frac {\log (2^{k-1} - 1) + \log 2 \cdot \log \gamma_k}{k-1-\alpha_{k,t}} - \log 2 \right) \right) \\
	&\qquad \overset{(\ref{eq: inner sum})}\leq re^{k-1} \sum_{h = r}^\infty \exp \left( h \cdot \left( \frac {3 \log r}{(k-1-\alpha_{k,t})r} + \frac {\log (2^{k-1} - 1) + \log 2 \cdot \log \gamma_k}{k-1-\alpha_{k,t}} - \log 2 \right) \right) \\
	&\qquad \overset{\hphantom{(\ref{eq: inner sum})}}\leq re^{k-1} \sum_{h = r}^\infty e^{-qh} \\
	&\qquad \overset{\hphantom{(\ref{eq: inner sum})}}= re^{k-1} \frac {e^{-qr}}{1-e^{-q}}
	\,.
\end{align*}
In particular, multiplying by $3r^3$ yields $3r^4e^{k-1} \frac {e^{-qr}} {1-e^{-q}}$, which tends to $0$ as $n \to \infty$ since $q$ is positive and $r$ can be chosen sufficiently large. This finishes the proof.
\end{proof}

For larger polymers $S \in \PM_{Z,b(n)}$, we will try to guarantee a large matching in the $(k-1)$-partite $(k-1)$-graph $L_G(S)$. For a partition class $Z$ of a $k$-partite $k$-graph $G$, $b \in \NN$, and $S \in \PM_{Z,b}$, we let $m(S)$ be the number of edges in a maximum matching of $L_G(S)$. Additionally for $s \in \NN$, let
\begin{align}\label{eq: matching number}
	m_{Z,b}(s) 
	\coloneqq \min_{\substack{S \in \PM_{Z,b} \\ \abs S = s}} m(S)
	\,.
\end{align}

\begin{lemma} \label{lem: expansion guarantees linear matching}
Let $k \geq 3$, $\beta > 0$, and $G$ be an $r$-regular $k$-partite $k$-graph. Then for every partition class $Z$ and every subset $S \subset Z$ with $\abs {N(S)} \geq (k-2 + \beta)r \abs S$, we have $m(S) \geq \frac \beta{k-1} r \abs S$.
\end{lemma}

\begin{proof}
Fix any matching $M$ of size $m(S)$ in $L_G(S)$. By maximality, each of the at most $r \abs S$ edges in $L_G(S)$ must intersect $M$ and therefore contains at most $k-2$ vertices from $N(S) \setminus V(M)$. Since every vertex in $N(S)$ is contained in at least one edge in $L_G(S)$ by definition, we find
\begin{align}\label{eq: edges by matching approach}
	r \abs S
	\geq \abs {E(L_G(S))}
	\geq \frac {\abs {N(S) \setminus V(M)}}{k-2}
	= \left( \frac 1{k-2} - \frac {(k-1)m(S)}{(k-2)\abs {N(S)}} \right) \abs {N(S)}
	\,.
\end{align}
Plugging this into our assumption of $\abs {N(S)} \geq (k-2 + \beta)r \abs S$ and dividing by $\abs {N(S)}$, we obtain
\begin{align*}
	1
	\geq (k-2 + \beta) \frac {r \abs S}{\abs {N(S)}}
	&\overset{(\ref{eq: edges by matching approach})}\geq (k-2 + \beta) \left( \frac 1{k-2} - \frac {(k-1)m(S)}{(k-2)\abs {N(S)}} \right) \\
	&\overset{\hphantom{(\ref{eq: edges by matching approach})}}\geq 1 + \frac \beta{k-2} - \frac {(k-1)m(S)}{(k-2)r \abs S}
	\,,
\end{align*}
where the last step uses $\abs {N(S)} \geq (k-2 + \beta)r \abs S$ again. Reordering yields $m(S) \geq \frac \beta{k-1} r \abs S$, as desired.
\end{proof}

Guaranteeing a large matching helps obtaining a bound by the following argument.

\begin{lemma}\label{lem: matching approach}
Let $k \geq 3$, $b \in \NN$, and $G$ be an $r$-regular $k$-partite $k$-graph. Then for every partition class $Z$ and every $S \in \PM_{Z,b}$, we have $w(S) \leq \gamma_k^{-m_{Z,b}(\abs S)}$.
\end{lemma}

\begin{proof}
Recall that deleting edges from a hypergraph does not decrease the number of independent sets. Fixing any maximum matching of $L_G(S)$ and deleting all edges not contained in this matching thus shows that
\begin{align}\label{eq: indepsets link graph by matching approach}
	\abs {\II(L_G(S))} 
	\leq (2^{k-1}-1)^{m(S)} \cdot 2^{\abs {N(S)} - (k-1)m(S)}
	= 2^{\abs {N(S)}} \cdot \gamma_k^{-m(S)}
	\,.
\end{align}
Using the definitions of $w(S)$ and $m_{Z,b}$, we immediately obtain
\[
	w(S)
	\overset{(\ref{eq: polymer weights})}= \abs {\II(L_G(S))} \cdot 2^{- \abs {N(S)}}
	\overset{(\ref{eq: indepsets link graph by matching approach})}\leq \gamma_k^{-m(S)}
	\overset{(\ref{eq: matching number})}\leq \gamma_k^{-m_{Z,b}(\abs S)}
	\,,
\]
as desired.
\end{proof}

We now apply~\cref{lem: matching approach} to all $S \in \PM_{Z,b(n)}$ not covered by \cref{lem: S polynomial controllable}, but split this into two cases at the threshold $\abs S = \beta \frac nr$. If $\abs S$ is below this threshold, the strong expansion guaranteed by \ExpL{\beta} allows for a straightforward argument, bounding the number of such polymers with \cref{lem: 2-linked fixed size}.

\begin{lemma}\label{lem: below delta n/r}
Let $k \geq 3$ and $\beta > 0$. Furthermore, let $b(n)$ be a function with $b(n) = o(n)$. Then there is $n_0 \in \NN$ such that the following holds for every $n \geq n_0$:

Let $G$ be an $r$-regular $k$-partite $k$-graph with vertex partition $\mathcal Z$, where each part has order $n$. If $G$ satisfies \Reg{t} and \ExpL{\beta}, then for every $u \in Z \in \mathcal Z$,
\[
	\sum_{s = r}^{\beta \frac nr} \sum_{\substack{S \in \PM_{Z,b(n)} \\ S \ni u \\ \abs S = s}} w(S) \exp \left( \frac {(k-1) s}r + \log \gamma_k \cdot r \log (2s) \right)
	\leq \frac 1{3r^3}
	\,.
\]
\end{lemma}

\begin{proof}
Let $r \leq s \leq \beta \frac nr$ and $S \in \PM_{Z,b(n)}$ with $\abs S = s$ be arbitrary. Since $G$ is an $r$\nobreakdash-regular $k$\nobreakdash-partite $k$-graph that satisfies \ExpL{\beta}, we can apply \cref{lem: expansion guarantees linear matching} and find that $m(S) \geq \frac \beta{k-1} rs$. As $S \in \PM_{Z,b(n)}$ was arbitrary, we conclude that $m_{Z,b(n)}(s) \geq \frac \beta{k-1}rs$. Plugging this into \cref{lem: matching approach}, we can bound the weight of any such polymer $S$ by
\begin{align}\label{eq: matching approach conclusion below delta n/r}
	w(S)
	\overset{L\ref{lem: matching approach}}\leq \gamma_k^{- m_{Z,b(n)}(s)}
	\leq \exp \left( -\frac {\beta \log \gamma_k}{k-1} rs \right)
	\,.
\end{align}

Completely analogous to (\ref{eq: number of polymers}) in the proof of \cref{lem: S polynomial controllable}, we again use \cref{lem: 2-linked fixed size} to observe that
\begin{align}\tag{\ref{eq: number of polymers}}
	\abs {\{ S \in \PM_{Z,b(n)} \colon \text{$S \ni u$ and $\abs S = s$} \}}
	\leq e((k-1)er^2)^{s-1}
	\leq r^{3s}
	= \exp \left( 3s \log r \right)
	\,.
\end{align}
Altogether, this implies that
\begin{align}\label{eq: main estimate below delta n/r}
	\notag &\sum_{\substack{S \in \PM_{Z,b(n)} \\ S \ni u \\ \abs S = s}} w(S) \exp \left( \frac {(k-1) s}r + \log \gamma_k \cdot r \log (2s) \right) \\
	\notag &\qquad \overset{\substack{(\ref{eq: number of polymers}) \\ (\ref{eq: matching approach conclusion below delta n/r})}}\leq \exp \left( 3s \log r - \frac {\beta \log \gamma_k}{k-1} rs + \frac {(k-1)s}r + \log \gamma_k \cdot r \log (2s) \right) \\
	&\qquad \overset{\hphantom{\substack{(\ref{eq: number of polymers}) \\ (\ref{eq: matching approach conclusion below delta n/r})}}}= \exp \left( s \cdot \left( 3 \log r - \frac {\beta \log \gamma_k}{k-1} r + \frac {k-1}r + \log \gamma_k \cdot r \frac {\log (2s)}s \right) \right)
	\,.
\end{align}

By choosing $n$ and thus $s \geq r \geq \frac 1t \log_{\gamma_k} n$ by \Reg{t} sufficiently large, we can ensure that the negative term in (\ref{eq: main estimate below delta n/r}) dominates the rest and we conclude
\begin{align*}
	\notag &\sum_{s = r}^{\beta \frac nr} \sum_{\substack{S \in \PM_{Z,b(n)} \\ S \ni u \\ \abs S = s}} w(S) \exp \left( \frac {(k-1) s}r + \log \gamma_k \cdot r \log (2s) \right) \\
	\notag &\qquad \overset{(\ref{eq: main estimate below delta n/r})}\leq n \exp \left( s \cdot \left( 3 \log r - \frac {\beta \log \gamma_k}{k-1} r + \frac {k-1}r + \log \gamma_k \cdot r \frac {\log (2s)}s \right) \right) \\
	&\qquad \overset{\hphantom{(\ref{eq: main estimate below delta n/r})}}\leq n \exp \left( - \frac {\beta \log \gamma_k}{2(k-1)} rs \right)
	\,.
\end{align*}
Now, we can choose $n$ large enough to ensure $s \geq \frac {4(k-1)t}\beta$ and use $r \geq \frac 1t \log_{\gamma_k} n$ to observe that
\[
	n \exp \left( - \frac {\beta \log \gamma_k}{2(k-1)} rs \right)
	\leq n \gamma_k^{-2rt}
	\leq \gamma_k^{-rt}
	\leq \frac 1{3r^3}
\]
for sufficiently large $r$. This finishes the proof.
\end{proof}

Above $\abs S = \beta \frac nr$, the guaranteed expansion is significantly weaker because $N(S)$ starts to take up a non-negligible fraction of $\overline Z$. This makes the $2$-linkedness of $S$ practically irrelevant. In this regime, the standard binomial upper bound on the number of $S \subset Z$ with $\abs S = s$ is actually just better than using \cref{lem: 2-linked fixed size} and (\ref{eq: number of polymers}).

\begin{lemma}\label{lem: above delta n/r}
Let $k \geq 3$ and $\beta > 0$. Furthermore, let $b(n)$ be a function with $b(n) = o(n)$. Then there is $n_0 \in \NN$ such that the following holds for every $n \geq n_0$:

Let $G$ be an $r$-regular $k$-partite $k$-graph with vertex partition $\mathcal Z$, where each part has order $n$. If $G$ satisfies \Reg{t}, \ExpM{\alpha_{k,t}}, and \ExpL{\beta}, then for every $u \in Z \in \mathcal Z$,
\[
	\sum_{s = \beta \frac nr}^{b(n)} \sum_{\substack{S \in \PM_{Z,b(n)} \\ S \ni u \\ \abs S = s}} w(S) \exp \left( \frac {(k-1) s}r + \log \gamma_k \cdot r \log (2s) \right)
	\leq \frac 1{3r^3}
	\,.
\]
\end{lemma}

\begin{proof}
Let $\beta \frac nr \leq s \leq b(n)$ and $S \in \PM_{Z,b(n)}$ with $\abs S = s$ be arbitrary. Also select an arbitrary subset $S' \subset S$ of order $\abs {S'} = \beta \frac nr$ and note that any matching in $S'$ is also a matching in $S$. Since $G$ is an $r$-regular $k$-partite $k$-graph that satisfies \ExpL{\beta}, we can apply \cref{lem: expansion guarantees linear matching} to $S'$ and find that $m(S) \geq m(S') \geq \frac \beta{k-1} r \abs {S'} = \frac {\beta^2}{k-1} n$. As $S \in \PM_{Z,b(n)}$ was arbitrary, we conclude that $m_{Z,b(n)}(s) \geq \frac {\beta^2}{k-1} n$. Plugging this into \cref{lem: matching approach}, we can bound the weight of any such polymer $S$ by
\begin{align}\label{eq: matching approach conclusion above delta n/r}
	w(S)
	\overset{L\ref{lem: matching approach}}\leq \gamma_k^{- m_{Z,b(n)}(s)}
	\leq \exp \left( -\frac {\beta^2 \log \gamma_k}{k-1} n \right)
	\,.
\end{align}

As hinted at already, we can bound the number of such polymers $S$ by the trivial binomial bound of 
\begin{align}\label{eq: number of polymers above delta n/r}
	\abs {\{ S \in \PM_{Z,b(n)} \colon \text{$S \ni u$ and $\abs S = s$} \}}
	\leq \binom {n}{s}
	\leq \left( \frac {en}s \right)^s
	= \exp \left( s + s \log \left( \frac ns \right) \right)
	\,.
\end{align}
Altogether, this implies that
\begin{align}\label{eq: main estimate above delta n/r}
	\notag &\sum_{\substack{S \in \PM_{Z,b(n)} \\ S \ni u \\ \abs S = s}} w(S) \exp \left( \frac {(k-1) s}r + \log \gamma_k \cdot r \log (2s) \right) \\
	\notag &\qquad \overset{\substack{(\ref{eq: number of polymers above delta n/r}) \\ (\ref{eq: matching approach conclusion above delta n/r})}}\leq \exp \left( s + s \log \left( \frac ns \right) - \frac {\beta^2 \log \gamma_k}{k-1} n + \frac {(k-1)s}r + \log \gamma_k \cdot r \log (2s) \right) \\
	&\qquad \overset{\hphantom{(\ref{eq: matching approach conclusion above delta n/r})}}= \exp \left( s \cdot \left( 1 + \log \left( \frac ns \right) - \frac {\beta^2 \log \gamma_k \cdot n}{(k-1)s} + \frac {k-1}r + \log \gamma_k \cdot r \frac {\log (2s)}s \right) \right)
	\,.
\end{align}

Note now that $s \leq b(n) = o(n)$, so by choosing $n$ sufficiently large, we can ensure that $\frac ns$ dominates $\log \left( \frac ns \right)$. Moreover, since $G$ also satisfies $\ExpM{\alpha_{k,t}}$ with $0 < \alpha_{k,t} < 1$, \cref{lem: ExpM saves KP3} guarantees that $r \leq \sqrt {2n}$ and consequently, $\frac ns$ also dominates $\frac {r \log (2s)}s \leq \frac {\sqrt {2n} \log (2n)}s$. In total, we conclude
\begin{align*}
	\notag &\sum_{s = \beta \frac nr}^{b(n)} \sum_{\substack{S \in \PM_{Z,b(n)} \\ S \ni u \\ \abs S = s}} w(S) \exp \left( \frac {(k-1) s}r + \log \gamma_k \cdot r \log (2s) \right) \\
	\notag &\qquad \overset{(\ref{eq: main estimate above delta n/r})}\leq n \exp \left( s \cdot \left( 1 + \log \left( \frac ns \right) - \frac {\beta^2 \log \gamma_k \cdot n}{(k-1)s} + \frac {k-1}r + \log \gamma_k \cdot r \frac {\log (2s)}s \right) \right) \\
	&\qquad \overset{\hphantom{(\ref{eq: main estimate below delta n/r})}}\leq n \exp \left( - \frac {\beta^2 \log \gamma_k}{2(k-1)} n \right)
	\,.
\end{align*}
More specifically, this term tends to $0$ as $n$ increases, even when multiplied by any fixed polynomial in $n$. We can thus use $r \leq \sqrt {2n}$ by \ExpM{\alpha_{k,t}} and \cref{lem: ExpM saves KP3} again to observe that
\[
	n \exp \left( - \frac {\beta^2 \log \gamma_k}{2(k-1)} n \right)
	\leq \frac 1{3 \sqrt {8´n^3}}
	\leq \frac 1{3r^3}
\]
for sufficiently large $n$. This finishes the proof.
\end{proof}

\section{Concluding remarks}
We first remark that the assumption of $k \geq 3$ in \cref{thm: main} is owed to the fact that our approach of guaranteeing a large matching in the relevant link graph is not well-defined if $G$ itself is already a $2$-graph. In this case of $k = 2$, however, existing literature already provides a similar statement, usually phrased as the existence of an FPTAS which approximates the number of independent sets in bipartite expander graphs by truncating the cluster expansion. In Theorem~1 of~\cite{JPP23}, for example, the only requirement we need to check is that $G$ is an \emph{$\varepsilon$-expander} for some small $\varepsilon$, i.e.\ $\abs {N(S)} \geq (1 + \varepsilon)\abs S$ for all $S \subset Z \in \mathcal Z$ with $\abs S \leq \frac n2$. For $\abs S \leq \frac {\beta^2 n}{1 + \varepsilon}$, this immediately follows from applying property \ExpL{\beta} to $S$ or a subset $S' \subset S$ of order $\abs {S'} = \beta \frac nr$. For $\abs S > \frac {\beta^2 n}{1 + \varepsilon} > b(n)$, we observe that $S \cup (\overline Z \setminus N(S))$ is an independent set and must therefore satisfy $\abs {\overline Z \setminus N(S)} \leq b(n)$ by property $\Def{b(n)}$, so $\abs {N(S)} \geq n - b(n) > (1 + \varepsilon)\frac n2 \geq (1 + \varepsilon)\abs S$ by $b(n) = o(n)$.

Secondly, note that we actually obtain slightly better bounds than the ones stated in \cref{thm: main}. In fact, careful consideration of the initial steps in the proofs of \cref{lem: split into polymer models,thm: main} shows that for any $\rho > 0$, the equality can be split into
\begin{align*}
	\abs {\II(G)}
	&\geq 2^{(k-1)n} \cdot \sum_{Z \in \mathcal Z} \exp \left( \sum_{m = 1}^t \sum_{\substack{\Gamma \in \CM_{Z,t} \\ \norm \Gamma = m}} w(\Gamma) - \rho n \gamma_k^{-rt} \right) - 2^{(k-2+\rho)n} \text{ and}\\
	\abs {\II(G)}
	&\leq 2^{(k-1)n} \cdot \sum_{Z \in \mathcal Z} \exp \left( \sum_{m = 1}^t \sum_{\substack{\Gamma \in \CM_{Z,t} \\ \norm \Gamma = m}} w(\Gamma) + \rho n \gamma_k^{-rt} \right)
	\,.
\end{align*}

Several directions seem interesting to us for future research. First and foremost, our result begs the question if and by how much the conditions imposed on our hypergraphs can be relaxed, while still allowing for asymptotic enumeration of independent sets with the cluster expansion method. 

Secondly, we see no reason why this approach should be limited to the mere number of independent sets. Basically, every variation that can still express its object of interest as the partition function of a suitable polymer model should be a good candidate for obtaining a similar result. Introducing a fugacity parameter $\lambda > 0$, one could for example consider the weighted sum
\[
	\sum_{I \in \II(G)} \lambda^{\abs I}
	\,.
\]
This would include our result at $\lambda = 1$. Even more generally, one could enumerate homomorphisms from regular $k$-partite $k$-graphs $G$ into a fixed graph $H$, very much akin to the impressive study in~\cite{JK20}. In particular, this would immediately provide information about the asymptotic number of proper colourings. 

A more difficult, albeit equally interesting question is whether our result generalizes to non-partite regular $k$-graphs. Here, it seems much harder to set up a similar counting scheme by identifying defect sets. This mirrors the complications arising in the graph case when trying to drop the bipartiteness requirement.


\bibliographystyle{amsplain}
\bibliography{indepsets.bib}

\providecommand{\bysame}{\leavevmode\hbox to3em{\hrulefill}\thinspace}
\providecommand{\MR}{\relax\ifhmode\unskip\space\fi MR }
\providecommand{\MRhref}[2]{%
  \href{http://www.ams.org/mathscinet-getitem?mr=#1}{#2}
}
\providecommand{\href}[2]{#2}
\begin{thebibliography}{10}

\bibitem{AJ23}
P.~Arras and F.~Joos, \emph{Independent sets in discrete tori of odd
  sidelength}, arXiv:2310.02747 (2023).

\bibitem{BS18}
A.~E. Balobanov and D.~A. Shabanov, \emph{On the number of independent sets in
  simple hypergraphs}, Mat. Zametki \textbf{1} (2018), 38--48, Russian;
  translation in \textit{Math. Notes}, \textbf{1-2} (2018), 33–41.

\bibitem{BGG+19}
I.~Bez\'{a}kov\'{a}, A.~Galanis, L.~A. Goldberg, H.~Guo, and
  D.~\v{S}tefankovi\v{c}, \emph{Approximation via correlation decay when strong
  spatial mixing fails}, SIAM J. Comput. \textbf{48} (2019), 279--349.

\bibitem{CP20}
S.~Cannon and W.~Perkins, \emph{Counting independent sets in unbalanced
  bipartite graphs}, Proceedings of the 2020 ACM-SIAM Symposium on Discrete
  Algorithms (SODA), SIAM, Philadelphia, PA, 2020, pp.~1456--1466.

\bibitem{CDK+23}
C.~Carlson, E.~Davies, A.~Kolla, and A.~Potukuchi, \emph{Approximately counting
  independent sets in dense bipartite graphs via subspace enumeration},
  arXiv:2307.09533 (2023).

\bibitem{CGG+21}
Z.~Chen, A.~Galanis, L.~A. Goldberg, W.~Perkins, J.~Stewart, and E.~Vigoda,
  \emph{Fast algorithms at low temperatures via {M}arkov chains}, Random
  Structures Algorithms \textbf{58} (2021), 294--321.

\bibitem{CGS+22}
Z.~Chen, A.~Galanis, D.~{\v{S}}tefankovi{\v{c}}, and E.~Vigoda, \emph{Sampling
  colorings and independent sets of random regular bipartite graphs in the
  non-uniqueness region}, arXiv:2105.01784 (2021).

\bibitem{CPS+22}
E.~Cohen, W.~Perkins, M.~Sarantis, and P.~Tetali, \emph{On the number of
  independent sets in uniform, regular, linear hypergraphs}, European J.
  Combin. \textbf{99} (2022), Paper No. 103401, 22 pages.

\bibitem{FGW+22}
W.~Feng, H.~Guo, C.~Wang, J.~Wang, and Y.~Yin, \emph{Towards derandomising
  {M}arkov chain {M}onte {C}arlo}, arXiv:2211.03487 (2022).

\bibitem{FV17}
S.~Friedli and Y.~Velenik, \emph{Statistical mechanics of lattice systems: A
  concrete mathematical introduction}, Cambridge University Press, 2017.

\bibitem{FGK+23}
T.~Friedrich, A.~Göbel, M.~S. Krejca, and M.~Pappik, \emph{Polymer dynamics
  via cliques: New conditions for approximations}, Theoret. Comput. Sci.
  \textbf{942} (2023), 230--252.

\bibitem{GGS21}
A.~Galanis, L.~A. Goldberg, and J.~Stewart, \emph{Fast algorithms for general
  spin systems on bipartite expanders}, ACM Trans. Comput. Theory \textbf{13}
  (2021), Art. 25, 18 pages.

\bibitem{GGS22}
A.~Galanis, L.~A. Goldberg, and J.~Stewart, \emph{Fast mixing via polymers for
  random graphs with unbounded degree}, Inform. and Comput. \textbf{285}
  (2022), Paper No. 104894, 16 pages.

\bibitem{Gal19}
D.~Galvin, \emph{Independent sets in the discrete hypercube}, arXiv:1901.01991
  (2019).

\bibitem{HWY23}
K.~He, C.~Wang, and Y.~Yin, \emph{Deterministic counting {L}ov{\'a}sz local
  lemma beyond linear programming}, Proceedings of the 2023 Annual ACM-SIAM
  Symposium on Discrete Algorithms (SODA), SIAM, Philadelphia, PA, 2023,
  pp.~3388--3425.

\bibitem{HPR20}
T.~Helmuth, W.~Perkins, and G.~Regts, \emph{Algorithmic {P}irogov-{S}inai
  theory}, Probab. Theory Related Fields \textbf{176} (2020), 851--895.

\bibitem{JK20}
M.~Jenssen and P.~Keevash, \emph{Homomorphisms from the torus}, Adv. Math.
  \textbf{430} (2023), Paper No. 109212, 89 pages.

\bibitem{JKP20}
M.~Jenssen, P.~Keevash, and W.~Perkins, \emph{Algorithms for \#{BIS}-hard
  problems on expander graphs}, SIAM J. Comput. \textbf{49} (2020), 681--710.

\bibitem{JMP24}
M.~Jenssen, A.~Malekshahian, and J.~Park, \emph{A refined graph container lemma
  and applications to the hard-core model on bipartite expanders},
  arXiv:2411.03393 (2024).

\bibitem{JPP23}
M.~Jenssen, W.~Perkins, and A.~Potukuchi, \emph{Approximately counting
  independent sets in bipartite graphs via graph containers}, Random Structures
  Algorithms \textbf{63} (2023), 215--241.

\bibitem{Kah01}
J.~Kahn, \emph{An entropy approach to the hard-core model on bipartite graphs},
  Combin. Probab. Comput. \textbf{10} (2001), 219--237.

\bibitem{KP20}
J.~Kahn and J.~Park, \emph{The number of 4-colorings of the {H}amming cube},
  Israel J. Math. \textbf{236} (2020), 629--649.

\bibitem{KS83}
A.~D. Korshunov and A.~A. Sapozhenko, \emph{The number of binary codes with
  distance {$2$}}, Problemy Kibernet. \textbf{40} (1983), 111--130, Russian.

\bibitem{KP86}
R.~Koteck{\'y} and D.~Preiss, \emph{Cluster expansion for abstract polymer
  models}, Comm. Math. Phys. \textbf{103} (1986), 491--498.

\bibitem{LLL+19}
C.~Liao, J.~Lin, P.~Lu, and Z.~Mao, \emph{Counting independent sets and
  colorings on random regular bipartite graphs}, arXiv:1903.07531 (2019).

\bibitem{LLL+22}
\bysame, \emph{An {FPTAS} for the hardcore model on random regular bipartite
  graphs}, Theoret. Comput. Sci. \textbf{929} (2022), 174--190.

\bibitem{PS23}
R.~Peled and Y.~Spinka, \emph{Rigidity of proper colorings of
  {${\mathbb{Z}}^d$}}, Invent. Math \textbf{232} (2023), 79--162.

\bibitem{Sap89}
A.~A. Sapozhenko, \emph{The number of antichains in ranked partially ordered
  sets}, Diskret. Mat. \textbf{1} (1989), 74--93, Russian; translation in
  \textit{Discrete Math. Appl.} \textbf{1} (1991), 35–58.

\end{thebibliography}

\bigskip

\section*{Appendix}
We prove \cref{lem: high girth determines local structure} and thus demonstrate how having high girth influences the local structure of a linear $k$-graph. The key observation is the following technical lemma.

\begin{lemma}\label{lem: high girth technical}
Let $k, m \geq 3$, $G$ be a linear $k$-partite $k$-graph, and $v_1, \ldots, v_{(k-1)m}$ be a sequence of vertices of $G$ such that the sets $e_i \coloneqq \{ v_{(k-1)(i-1)+1}, \ldots, v_{(k-1)i+1} \}$ for $i \in \{ 1, \ldots, m \}$ are distinct edges of $G$, where subscripts are taken modulo $(k-1)m$. Then $G$ has girth at most $m$.
\end{lemma}

\begin{proof}
We induct on $m$ and start with the induction step for $m \geq 4$ arbitrary. If all vertices in $C \coloneqq v_1, \ldots, v_{(k-1)m}$ are distinct, then $C$ itself constitutes a loose $m$-cycle in $G$ and there is nothing to show. We can therefore assume that there are two indices $j \neq j' \in \{ 1, \ldots, (k-1)m \}$ such that $v_j = v_{j'}$.

For any $j \in \{ 2, \ldots, (k-1)m \}$, we denote the indices of the edges that position $j$ is involved in as
\[
	I(j)
	\coloneqq \{ i \in \{ 1, \ldots, m \} \colon (k-1)(i-1)+1 \leq j \leq (k-1)i+1 \}
	\,.
\]
So $\abs {I(j)} \in \{ 1, 2 \}$ for every $j$. Naturally, we also set $I(1) \coloneqq \{ 1, m \}$. For two indices $j, j' \in \{ 1, \ldots, (k-1)m \}$, we further let their distance in $C$ be
\begin{align}\label{eq: index distance}
	d(j,j')
	\coloneqq \min_{\substack{i \in I(j) \\ i' \in I(j')}} \min \{ i' - i, i - i' \}
	\,,
\end{align}
where the differences $i'-i, i-i' \in \{ 0, \ldots, m-1 \}$ are taken modulo $m$. Out of all pairs $j \neq j' \in \{ 1, \ldots, (k-1)m \}$ with $v_j = v_{j'}$, we fix one minimizing $d(j,j')$ and let $i \in I(j), i' \in I(j')$ be the choices yielding said minimum distance in (\ref{eq: index distance}). Reversing and cyclically shifting the sequence $C$ as needed, we can assume $i' = 1 \leq i$ and $i = 1 + d(j,j')$ without loss of generality.

Next, we show that $d(j,j') \geq 2$ by contradiction. If $d(j,j') = 0$, then $i = 1$ and $e_1 \in E(G)$ would contain less than $k$ vertices. This contradicts $G$ being a $k$-graph. We conclude that $d(j,j') \geq 1$ and $i \geq 2$. In particular, this implies that $j' < k$, otherwise $i' = 2$ would yield a smaller $d(j,j')$. Similarly, $j > (k-1)(i-1)+1$, otherwise $i-1$ would yield a smaller $d(j,j')$.

If $d(j,j') = 1$, then $i = 2$ and $j > k$. Also, $v_j \neq v_k$, else $d(j,k) = 0 < 1 = d(j,j')$. However, this implies that $e_1 \cap e_2 \supseteq \{ v_j, v_k \}$, which contradicts $G$ being linear. We conclude that indeed, $d(j,j) \geq 2$ and thus, $3 \leq i \leq m-1$. 

Recall that $j' < k$ and $j > (k-1)(i-1)+1$. We now consider the following sequence:
\[
	C' 
	\coloneqq v_{j'}, v_1, \ldots, v_{j'-1}, v_{j'+1}, \ldots v_{j-1}, v_{j+1}, \ldots, v_{(k-1)i+1}
	\,.
\]
It is derived from $C$ by reordering $e_1$ such that it starts with $v_{j'}$ and reordering $e_i$ such that it ends with $v_j$, leaving out $v_j$ and everything after $e_i$. Since $v_j = v_{j'}$ and $3 \leq i \leq m-1$, we can apply the induction hypothesis to $C'$. This finishes the proof of the induction step.

It remains to prove the induction base $m = 3$. Since the proof of the induction step does not use the fact that $m \geq 4$, we can again conclude that if there is a repetition $v_j = v_{j'}$ in a sequence $C$ of $3(k-1)$ vertices, then $d(j,j') \geq 2$. However, since the sequence $C$ only consists of $m = 3$ edges of $G$, any pair of positions $j,j'$ is at distance at most $1$. This shows that there cannot be a repetition and $C$ is indeed a loose $3$-cycle.
\end{proof}

Using this, we can easily establish \cref{lem: high girth determines local structure}, which states that vertices $v, u \in Z$ that are adjacent in $Z_2$ share exactly one neighbour in $G$.

\begin{proof}[Proof of \cref{lem: high girth determines local structure}]
We first note that $N_G(\{ v \})$ and $N_G(\{ u \})$ must intersect in at least one vertex $x \in \overline Z$ since $vu \in E(Z_2)$. For a proof of $\abs {N_G(\{ v \}) \cap N_G(\{ u \})} = 1$ by contradiction, suppose that there are distinct vertices $x, x' \in N_G(\{ v \}) \cap N_G(\{ u \})$. This implies that there are edges $e_{vx} \coloneqq \{ v, y_{vx}^{(1)}, \ldots, y_{vx}^{(k-2)}, x \}$, $e_{xu} \coloneqq \{ x, y_{xu}^{(1)}, \ldots, y_{xu}^{(k-2)}, u \}$, $e_{ux'} \coloneqq \{ u, y_{ux'}^{(1)}, \ldots, y_{ux'}^{(k-2)}, x' \}$, and $e_{x'v} \coloneqq \{ x', y_{x'v}^{(1)}, \ldots, y_{x'v}^{(k-2)}, v \}$ in $G$. We now consider the sequence
\[
	C_4 \coloneqq v, y_{vx}^{(1)}, \ldots, y_{vx}^{(k-2)}, x, y_{xu}^{(1)}, \ldots, y_{xu}^{(k-2)}, u, y_{ux'}^{(1)}, \ldots, y_{ux'}^{(k-2)}, x', y_{x'v}^{(1)}, \ldots, y_{x'v}^{(k-2)}
	\,.
\]

Since $v, u$ are distinct and the only vertices of $Z$ in $C_4$, the only possibilities for identical edges in $C_4$ are $e_{vx} = e_{x'v}$ and $e_{xu} = e_{ux'}$. If both were true, both edges would contain $x, x'$, which contradicts $G$ being linear. So at most one of them is true, without loss of generality $e_{xu} = e_{ux'}$. We then replace the subsequence
\[
	x, y_{xu}^{(1)}, \ldots, y_{xu}^{(k-2)}, u, y_{ux'}^{(1)}, \ldots, y_{ux'}^{(k-2)}, x'
\]
of $C_4$ by $x, C'', x'$ with $C''$ being an arbitrary ordering of $e_{xu} \setminus \{ x, x' \}$. Applying \cref{lem: high girth technical} to either this sequence (if $e_{vx} = e_{x'v}$ or $e_{xu} = e_{ux'}$) or $C_4$ itself (if neither $e_{vx} = e_{x'v}$ nor $e_{xu} = e_{ux'}$), we find that $G$ has girth at most $4$. This, however, contradicts our assumption that $G$ has girth at least $5$ and thus finishes the proof.
\end{proof}

\end{document}